\documentclass{amsart}
\usepackage{amssymb}
\usepackage{stmaryrd} 
\usepackage{amsmath} 
\usepackage{amscd}
\usepackage{amsthm}
\usepackage{amsbsy}
\usepackage[margin=1in]{geometry}
\usepackage{commath, longtable}
\usepackage{comment, enumerate}
\usepackage{geometry}
\usepackage[matrix,arrow]{xy}
\usepackage{hyperref}
\usepackage{mathrsfs}
\usepackage{color}
\usepackage{float}
\usepackage{mathtools,caption}
\usepackage[table,dvipsnames]{xcolor}
\usepackage{tikz-cd}
\usepackage{longtable}
\usepackage[utf8]{inputenc}
\usepackage[OT2,T1]{fontenc}
\usepackage[normalem]{ulem}
\usepackage{hyperref}
\usepackage[customcolors]{hf-tikz}
\usepackage{enumitem}
\newlist{primenumerate}{enumerate}{1}
\setlist[primenumerate,1]{label={\arabic*$'$}}
\hypersetup{
 colorlinks=true,
 linkcolor=DarkOrchid,
 filecolor=blue,
 citecolor=olive,
 urlcolor=orange,
 pdftitle={twist project},
 }
\usepackage{booktabs}

\DeclareSymbolFont{cyrletters}{OT2}{wncyr}{m}{n}
\DeclareMathSymbol{\Sha}{\mathalpha}{cyrletters}{"58}
\newcommand{\DK}[1]{\textcolor{purple}{#1}}
\newcommand{\KM}[1]{\textcolor{OliveGreen}{#1}}

\newtheorem{theorem}{Theorem}[section]
\newtheorem{lemma}[theorem]{Lemma}
\newtheorem{conj}[theorem]{Conjecture}
\newtheorem{proposition}[theorem]{Proposition}
\newtheorem{corollary}[theorem]{Corollary}
\newtheorem{definition}[theorem]{Definition}

\newtheorem*{theorem*}{Theorem}
\newtheorem{assumption}[theorem]{Assumption}
\numberwithin{equation}{section}
\newtheorem{lthm}{Theorem}

\theoremstyle{remark}
\newtheorem{remark}[theorem]{Remark}

\newcommand{\GL}{\operatorname{GL}}
\newcommand{\SL}{\operatorname{SL}}

\newcommand{\Image}{\operatorname{Im}}

\newcommand{\Gal}{\operatorname{Gal}}
\newcommand{\Frob}{\operatorname{Frob}}
\newcommand{\cyc}{\operatorname{cyc}}
\newcommand{\Sel}{\operatorname{Sel}}

\newcommand{\rk}{\operatorname{rank}}

\newcommand{\an}{\operatorname{an}}
\newcommand{\alg}{\operatorname{alg}}

\newcommand{\Gr}{\operatorname{Gr}}
\newcommand{\SelBDP}{\operatorname{Sel}^{\textup{BDP}}}
\newcommand{\Cl}{\operatorname{Cl}}

\newcommand{\Z}{\mathbb{Z}}
\newcommand{\Zp}{\mathbb{Z}_p}
\newcommand{\Qp}{\mathbb{Q}_p}
\newcommand{\Q}{\mathbb{Q}}
\newcommand{\EC}{\mathsf{E}}

\newcommand{\cL}{\mathcal{L}}

\newcommand{\fp}{\mathfrak{p}}

\makeatletter
\theoremstyle{plain} %makes the name of the conjecture/theorem bold
\newtheorem*{intr@thm}{\intr@thmname}

\newtheorem*{c@njecture}{\conjn@name}
\newcommand{\myl@bel}[2]{
 \protected@write \@auxout {}{\string \newlabel {#1}{{#2}{\thepage}{#2}{#1}{}} }
 \hypertarget{#1}{}
 } %showing an error here

\makeatother

\makeatletter
\newcommand{\mylabel}[2]{#2\def\@currentlabel{#2}\label{#1}}
\makeatother

\title{Iwasawa Theory of Elliptic Curves in Quadratic Twist Families}

\author[D.~Kundu]{Debanjana Kundu}
\address[Kundu]{University of Regina, Saskatchewan, Canada}
\email{debanjana.kundu@uregina.ca}

\author[K.~M\"uller]{Katharina M\"uller}
\address[Müller]{Institut für Theoretische Informatik, Mathematik und Operations Research, Universität der Bundeswehr München, Werner-Heisenberg-Weg 39, 85577 Neubiberg, Germany}
\email{katharina.mueller@unibw.de}

\date{\today}

\begin{document}

\begin{abstract}
In this article, we use two different approaches -- one algebraic and the other analytic -- to study the variation of Iwasawa invariants of rational elliptic curves in some quadratic twist families.
The analytic approach involves a thorough investigation of half-integral weight modular forms.
On the other hand, the algebraic proof requires studying the BDP-Selmer groups and the fine Selmer groups.
\end{abstract}


\maketitle

\section{Introduction}

\subsubsection*{Arithmetic of elliptic curves and their quadratic twists}
Given an elliptic curve $\EC$ defined over a number field $K$, the Mordell--Weil theorem asserts that the group of $K$-rational points, called the Mordell--Weil group and denoted by $\EC(K)$, is finitely generated as an abelian group (see \cite{Mor22, Wei29}).
A central question in the arithmetic of elliptic curves is the precise structure of this group, in particular its (algebraic) rank.
A long-standing conjecture in the area is the \emph{rank distribution conjecture} which claims that over any number field a suitably-defined \emph{average rank} would be 1/2.
More precisely, half of all elliptic curves have Mordell--Weil rank 0 and the remaining half have (Mordell--Weil) rank 1; whereas higher (Mordell–Weil) ranks constitute 0\% of all elliptic curves, even though there may exist infinitely many such elliptic curves.
The aforementioned rank distribution conjecture originated in the work of D.~Goldfeld; see \cite{Gol79} where he predicted the rank distribution in quadratic twist families of elliptic curves.
More precisely, let $\EC: y^2 = f(x)$ be a rational elliptic curve (where $f(x)$ is a cubic polynomial) and for $d$ a square-free integer (positive or negative) set $\EC^{(d)}: dy^2 = f(x)$ to be its quadratic twist.
Then the Mordell--Weil rank of $\EC^{(d)}/\Q$ is 0 (resp. 1) for half (resp. half) of all square free $d$'s and for 0\% of all square-free $d$'s, the (Mordell--Weil) rank of $\EC^{(d)}/\Q$ is at least 2.
Note that $\EC$ and $\EC^{(d)}$ are not isomorphic over $\Q$ but become isomorphic over the quadratic field $K = \Q(\sqrt{d})$.
Sometimes, depending on the context, the twisted curve is denoted as $\EC^K$.

The best results towards the rank distribution conjecture are by M.~Bharagava and A.~Shankar (see \cite{BS15_quartic, BS15_cubic}); they show that the average rank of elliptic curves over $\Q$ is strictly less than one, and that both rank 0 and rank 1 cases comprise non-zero densities across all rational elliptic curves.
In \cite{Smi25}, A.~Smith has announced a proof of Goldfeld's Conjecture which is conditional on the Birch--Swinnerton-Dyer Conjecture.
Proving results about the Mordell--Weil rank almost always involves a thorough analysis of the \emph{Selmer group}.
This idea dates back to the proof of the Mordell--Weil theorem, and is even an essential part of contemporary results like that of Bhargava--Shankar and Smith.

\subsubsection*{Iwasawa theory of elliptic curves}
Another subject of active research is Iwasawa theory.
A motivating problem in the Iwasawa theory of elliptic curves is the question of determining the growth of the Mordell--Weil rank in certain infinite towers of number fields such as the cyclotomic $\Zp$-extension.
Such questions were first studied by B.~Mazur in \cite{Maz72}, where he showed that for a class of elliptic curves defined over $\Q$, the Mordell--Weil rank remains bounded in the cyclotomic $\Z_p$-extension $\Q_{\cyc}/\Q$.
Mazur's proof (unsurprisingly) involved a thorough analysis of the $p$-primary Selmer group of $\EC/\Q_{\cyc}$, which is denoted by $\Sel(\EC/\Q_{\cyc})$.
This result has been extended to all \emph{rational} elliptic curves by K.~Kato \cite{Kat04} and D.~Rohrlich \cite{Roh88}.

Let $\EC/\Q$ be an elliptic curve with good \emph{ordinary} reduction at $p\geq 5$.
Over $\Q_{\cyc}$, it is known that the $p$-primary Selmer group is cofinitely generated as a module over the Iwasawa algebra (denote by $\Lambda$) and is co-torsion.
The Iwasawa algebra is isomorphic to the power series ring $\Zp\llbracket x \rrbracket$.
The algebraic structure of the Selmer group (as a $\Lambda$-module) is encoded by Iwasawa invariants, $\mu$ and $\lambda$.
By the $p$-adic Weierstrass Preparation Theorem, the characteristic ideal of (the Pontryagin dual of) the Selmer group is generated by a unique element $f^{(p)}_{\Sel(\EC/\Q_{\cyc})^\vee} (x)$, which can be expressed as a power of $p$ times a distinguished polynomial.
It is well-known that  the $\lambda$-invariant of $\Sel(\EC/\Q_{\cyc})^\vee$ is at least as large as the Mordell--Weil rank of $\EC/\Q$.
Thus, the study of the $\lambda$-invariant attracts significant attention when trying to gain information about the rank of the elliptic curve.

\subsubsection*{Motivation}
The Euler characteristic for $\EC$ is an invariant of the elliptic curve which can be computed (up to a $p$-adic unit) by knowing well-understood arithmetic invariants of $\EC/\Q$ but it (miraculously) encodes information on the Iwasawa invariants of $\Sel(\EC/\Q_{\cyc})^\vee$.
Using the Euler characteristic for $\EC$ it has been possible to study the variation of Iwasawa invariants in certain families of elliptic curves, such as all rank 0 elliptic curves over $\Q$ with good reduction at $p$.
In view of these results, it is expected that for a fixed prime $p$, the proportion of elliptic curves of rank 0 over $\Q$ with good ordinary reduction at $p$ (ordered by height) with trivial $p$-primary Selmer group (equivalently $\mu=\lambda=0$) approaches 100\% as $p \to \infty$; (see \cite[Conjecture~4.7]{KR21} and for refinements see \cite{KR24}).
These expectations hinge on assuming finiteness of the $p$-primary part of the Shafarevich--Tate group and on Cohen--Lenstra type heuristics on the variation of the Shafarevich--Tate group group by C.~Delaunay \cite{Del01}.
[Were it easier to compute the $p$-adic regulator of elliptic curves, one might have more ambitiously predicted that the proportion of $\EC/\Q$ of rank $1$ with good ordinary reduction at $p$ (ordered by height) with $\mu=0$ and $\lambda=1$ approaches 100\% as $p \to \infty$.]
Using the Euler characteristic to study variation of Iwasawa invariants in families of elliptic curves has been moderately successful in some other cases; for example, changing the base field and/or the $\Zp$-extension (see \cite{KLR21_noncomm, HKR}).
However, it has been almost impossible to use this idea to study the variation of Iwasawa invariants of elliptic curves in \emph{thin families} such as quadratic twist families of a fixed elliptic curve $\EC/\Q$.

We use two (different) approaches -- one algebraic and the other analytic -- to study the variation of Iwasawa invariants of elliptic curves in some (quadratic) twist families.
These ideas are detailed below.
To the best knowledge of the authors, it is a first attempt to use the theory of modular forms of half-integral weights to study Iwasawa invariants of elliptic curves.

\subsection{Main Results}

\subsubsection*{Approach I: Analytic}

The first main result we prove shows that in certain quadratic twist families the Iwasawa $\lambda$-invariant remains constant.
More precisely, we prove the following result.

\begin{lthm}
\label{thm A}
Let $\EC/\Q$ be an elliptic curve of square-free odd conductor {$N_{\EC}$} and $f$ be the associated modular form.
Let $F\in S_{3/2}(4N_{\EC},1,f)$ be a modular form of weight $3/2$ corresponding to $f$ under the Shimura lift.
Let $n_1, n_2$ be square free integers coprime to $p$ such that $\frac{n_1}{n_2}\in (\Q_q^\times)^2$ for all $q\mid 4N_{\EC}p$.
Assume that $\chi_{-n_1}$ and $\chi_{-n_2}$ have the same sign.
Assume that $a_{n_2}(F)$ is a $p$-adic unit and $\lambda_{\an}(\EC^{(-n_1)}/\Q_{\cyc})=\mu_{\an}(\EC^{(-n_1)}/\Q_{\cyc})=0$.
Then the same is true for $\EC^{(-n_2)}$.
Further, if the Iwasawa Main conjecture holds then the assertion is true for the algebraic Iwasawa invariants.
\end{lthm}

Comment on notation: We remind the reader that here, $\EC^{(-n)}$ denotes the quadratic twist of $\EC/\Q$ by $\Q(\sqrt{-n})$ for positive and square-free $n$.
Also, the analytic Iwasawa invariants are the ones associated with the $p$-adic $L$-function of the elliptic curve; see Section~\ref{sec: p-adic L function defn} for definitions.

The way to interpret this result is as follows: fix a square-free positive integer $n_{1}$ coprime to $p$ such that the analytic Iwasawa invariants of the twisted elliptic curve $\EC^{(-n_1)}$ are trivial.
Vary $n_2$ over square-free positive integer coprime to $p$ such that $a_{n_2}(F_{\EC})$ is a $p$-adic unit, where $F=F_{\EC}$ as above and $\frac{n_1}{n_2}\in (\Q_q^*)^{2}$ for all $q\mid 2N_{\EC}p$, and $\chi_{-n_1}$ and $\chi_{-n_2}$ have the same sign.
Then, the twisted elliptic curve $\EC^{(-n_{2})}$ has the same Iwasawa-invariants as that of the (fixed) elliptic curve $\EC^{(-n_{1})}$.
In this case, we are not proving results on $\lambda$-invariants being equal to the original elliptic curve $\EC$.

\subsubsection*{Approach II: Algebraic}

Using techniques from Galois cohomology, results on fine Selmer groups, and explicit twisting of elliptic curves, we prove the following result.

\begin{lthm}
\label{thm B}
Let $\EC/\Q$ be an elliptic curve.
Let $p\geq 5$ be a prime and $K$ be a quadratic field satisfying the following properties
    \begin{enumerate}
    \item[\textup{(}i\textup{)}] $\EC(K_v)[p]=\{0\}$ for all $v\in \Sigma(K)$ where $\Sigma(K)$ denotes the set of all primes in $K$ above the conductor of $\EC$, the primes above $p$, the archimedean prime(s), and the primes ramified in $K/\Q$,
    \item[\textup{(}ii\textup{)}] $\EC$ has good ordinary reduction at $p$,
    \item[\textup{(}iii\textup{)}] $p$ splits in $K$,
    \item[\textup{(}iv\textup{)}] the $p$-primary part of the fine Selmer group of $\EC/K$ is trivial,
    \item[\textup{(}v\textup{)}] Assume that $\EC/K$ is of rank $2$. 
\end{enumerate}
Then there is an isomorphism
\[
\Sel(\EC/\Q_{\cyc})^\vee\cong \Sel(\EC^K/\Q_{\cyc})^\vee
\]
as $\Lambda$-modules.
In particular, $\lambda_{\alg}(\EC/\Q_{\cyc}) = \lambda_{\alg}(\EC^{K}/\Q_{\cyc})$.
\end{lthm}

The proof of Theorem~\ref{thm B} is an assertion on the classical Selmer groups.
However, the new key idea of our proof is to invoke the study of the BDP Selmer group to get the desired information.
Our proof also requires a thorough analysis of the fine Selmer group which is a subgroup of the classical Selmer group.
In addition to the aforementioned theorem, under additional hypotheses on finiteness of the $p$-primary part of the Shafarevich--Tate group, we also prove using results of J.~Lee that for all $n$
\[
\rk_{\Z}(\EC(\Q_{(n)}))=\rk_{\Z}(\EC^K(\Q_{(n)})).
\]
In Theorem~\ref{thm 4.11} and Corollary~\ref{class number corollary} we discuss sufficient arithmetic conditions (which can be checked using SAGE/MAGMA for example) to verify the hypothesis.
In addition, for a fixed elliptic curve $\EC/\Q$ and a fixed imaginary quadratic field $K$, we can show for what proportion of primes the assumptions in Theorem~\ref{thm B} are satisfied; see Section~\ref{sec: fix K and E}.
Since the $\lambda$-invariant over $\Q_{\cyc}$ has the same parity as the Mordell--Weil rank of the rational elliptic curve (at least when assuming finiteness of the Shafarevich--Tate group) -- see \cite[Proposition~3.10]{greenberg-elliptic} -- the conclusion of Theorem~\ref{thm B} tells us that our hypotheses will in fact force the rank of $\EC/K$ to be even.
This in discussed in greater detail in Remark~\ref{rem:rank} and Section~\ref{sec: rank 0 discussion}.
%The complementary case where $\EC/K$ has odd rank is treated in Theorem~\ref{sha-trivial}.

Finally,  we compute for `how many' quadratic fields the above conditions hold (when $\EC/\Q$ and $p\geq 5$ are kept fixed); see Theorem~\ref{density result}.
Unfortunately, this is a `thin set' among all quadratic fields.
However, if we restrict the count to quadratic fields $K = \Q(\sqrt{\pm \ell})$ with $\ell$ a prime and then consider (appropriate) density, we get positive proportions that depend on the number of prime divisors and parity of $N_{\EC}$.
Recall that a conjecture of J.~Silverman predicts that there are infinitely many primes $\ell$ such that $\EC^K$ has rank zero (resp. 1) if $K$ varies over the fields $K=\Q(\sqrt{\ell})$ for primes $\ell$.
Thus, if $\EC/\Q$ is of rank $1$ we would expect that condition (v) is satisfied for a positive density of twist.

We emphasize that our main results do not require any hypothesis on the finiteness of the Shafarevich--Tate group over $\Q$ or $K$.
However, to study how often the hypotheses are satisfied we often need to make this finiteness assumption.

\subsection{Organization}
Including this introduction, the article has six sections.
Section~\ref{sec: preliminary} is preliminary in nature, where we recall basic definitions from Iwasawa theory and the theory elliptic curves.
We remind the reader of the Greenberg--Iwasawa Main Conjecture (IMC) -- which is known to be a theorem in many cases.
We also discuss the basic notions associated with modular forms of half-integral weight.
In Section~\ref{sec: approach 1} we prove results with an analytic flavour.
We can show that certain quadratic twist families of elliptic curves have constant (analytic) Iwasawa invariants over $\Q_{\cyc}$.
Using IMC, this result is converted to a statement on algebraic Iwasawa invariants.
In Section~\ref{sec: approach 2} we prove results which are algebraic in nature.
For a fixed elliptic curve $\EC/\Q$ we find quadratic twist families $\EC^{(n)}$ such that the (dual) Selmer groups of $\EC$ and $\EC^{(n)}$ are isomorphic as $\Lambda$-modules over $\Q_{\cyc}$.
In particular, the Iwasawa invariants in such quadratic twist families remain the same.
Under appropriate assumptions, we can also comment on the equality of Mordell--Weil ranks of $\EC$ and its twist in each layer of the cyclotomic $\Zp$-extension.
{Section~\ref{sec: verifying} is dedicated to showing that the hypotheses we invoke are reasonably mild.}
In Section~\ref{sec: density} we prove density results.

\subsection{Outlook}
For the analytic approach, we have restricted ourselves to the study of the variation of Iwasawa invariants associated to elliptic curves in \emph{imaginary quadratic} twist families.
One possibility is to consider real quadratic twist families, as well.
The analytic approach builds on a result of Waldspurger relating $L$-values of modular forms and imaginary quadratic twists \cite{waldspurger}.
To the knowledge of the authors no analogue for real quadratic twists is known.
Throughout this article, we have considered the case where $p$ is a prime of good \emph{ordinary} reduction for $\EC/\Q$.
We intend to study the case when $p$ is a prime of good \emph{supersingular} reduction in the future.
Finally, it seems well within reach that at least some of our results can be extended to the case of modular abelian varieties.

\section*{Acknowledgements}
We thank Antonio Lei for his questions during a talk that DK gave about this work at the University of Ottawa.
We thank Jeffrey Hatley, Sören Kleine, and Ahmed Matar for their comments on an earlier draft of this paper.
DK acknowledges the support of an AMS--Simons Early Career Travel grant and NSERC Discovery grant RGPIN-2026-07384.
{A first draft of this} project was completed while in residence at Simons Laufer Mathematical Sciences Institute (formerly MSRI) in Berkeley, California as part of the Summer Research in Mathematics (2025).
We thank SLMath for its hospitality and financial support.
This work was supported by the NSF grant DMS-1928930.

\section{Preliminaries}

\label{sec: preliminary}

\subsection{Iwasawa theory}

Fix an odd prime $p$.
For a non-negative integer $n$, write $\zeta_{p^n}$ to denote a primitive $p^n$-th root of unity.
Then the \emph{cyclotomic} $\Zp$-extension of $\Q$ is the maximal totally real pro-$p$ subfield of $\bigcup_n \Q(\zeta_{p^n})$.
We denote this extension by $\Q_{\cyc}$.
In particular, there is a tower
\[
\Q = \Q_{(0)} \subset \Q_{(1)} \subset \ldots \subset \Q_{\cyc}
\]
where each $\Q_{(n)}$ is the totally real subfield of $\Q(\zeta_{p^{n+1}})$ with $\Gal(\Q_{(n)}/\Q)\simeq \Z/p^n\Z$.
By infinite Galois theory,
\[
\Gamma = \Gal(\Q_{\cyc}/\Q) \simeq \varprojlim_n \Z/p^n\Z = \Zp.
\]
Let $K$ be any number field.
Then $K_{\cyc} = K \cdot \Q_{\cyc}$, i.e., the compositum of $K$ with $\Q_{\cyc}$.
As before, we write $K_{(n)}$ to denote the $n$-th layer of $K_{\cyc}$ and write $\Gamma_n = \Gal(K_{(n)}/K) \simeq \Z/p^n \Z$.
By definition $\Gal(K_{\cyc}/K)\simeq \Zp$, as well.

The \emph{Iwasawa algebra} $\Lambda=\Lambda(\Gamma)$ is the completed group algebra $\Z_p\llbracket \Gamma \rrbracket :=\varprojlim_n \Z_p[\Gamma/\Gamma^{p^n}]$.
Fix a topological generator $\gamma$ of $\Gamma$; this gives an isomorphism of rings 
\begin{align*}
\Lambda &\xrightarrow{\sim} \Z_p\llbracket x\rrbracket \\
\gamma & \mapsto 1+x.
\end{align*}

Let $M$ be a finitely generated torsion $\Lambda$-module.
The \emph{Structure Theorem of $\Lambda$-modules} asserts \cite[Theorem~13.12]{Was97} that $M$ is pseudo-isomorphic to a finite direct sum of cyclic $\Lambda$-modules.
In other words, there is a homomorphism of $\Lambda$-modules
\[
M \longrightarrow \left(\bigoplus_{i=1}^s \Lambda/(p^{m_i})\right)\oplus \left(\bigoplus_{j=1}^t \Lambda/(f_j(x)) \right)
\]
with finite kernel and cokernel.
Here, $m_i>0$ and $f_j(x)$ is a distinguished polynomial (i.e. a monic polynomial with non-leading coefficients divisible by $p$).
The characteristic ideal of $M$ is (up to a unit) generated by the characteristic element,
\[
f_{M}^{(p)}(x) := p^{\sum_{i} m_i} \prod_j f_j(x).
\]
The $\mu$-invariant of $M$ is defined as the power of $p$ in $f_{M}^{(p)}(x)$.
More explicitly,
\[
\mu_{\alg}(M) = \mu_p(M):=\begin{cases}0 & \textrm{ if } s=0\\
\sum_{i=1}^s m_i & \textrm{ if } s>0.
\end{cases}
\]
The $\lambda$-invariant of $M$ is the degree of the characteristic element, i.e.
\[
\lambda_{\alg}(M) = \lambda_p(M) := \sum_{j=1}^t \deg f_j(x).
\]

\subsection{Selmer group and fine Selmer group of elliptic curves}
Fix an algebraic closure $\overline{\Q}$ of $\Q$.
Then an algebraic extension of $\Q$ is a subfield of this fixed algebraic closure, $\overline{\Q}$.
Throughout, $p$ denotes a fixed rational odd prime.
Let $\EC/\Q$ be an elliptic curve.
Fix a finite set $\Sigma = \Sigma(\Q)$ of primes of $\Q$ \emph{containing} $p$ and the primes of bad reduction of $\EC$.
Denote by $\Q_{\Sigma}$, the maximal algebraic extension of $\Q$ unramified outside $\Sigma$.
For every (possibly infinite) extension $L$ of $\Q$ contained in $\Q_{\Sigma}$, write $G_{\Sigma}\left({L}\right) = \Gal\left(\Q_{\Sigma}/{L}\right)$.
Write $\Sigma(L)$ for the set of primes of $L$ above $\Sigma$.
{For any algebraic extension $L$ of $\Q$, we write $\Sigma_p(L)$ to denote denote the primes above $p$ in $L$.}
If $L$ is a finite extension of $\Q$ and $w$ is a place of $L$, we write $L_w$ for its completion at $w$; when $L/\Q$ is infinite, it is the union of completions of all finite sub-extensions of $L$.

\begin{definition}
Fix an elliptic curve $\EC/\Q$ with good reduction at $p$.
View $\EC[p^\infty]$ as a $\Zp$-module equipped with a continuous $G_{\Sigma}(\Q)$-action.
\begin{itemize}
 \item[\textup{(}i\textup{)}] For any finite extension $L/\Q$ set
\begin{align*}
K^1_v\left(\EC/L\right) &= \bigoplus_{w|v} H^1\left(L_w, \EC[p^\infty]\right),\\
J^1_v\left(\EC/L\right) &= \bigoplus_{w|v} H^1\left(L_w, \EC\right)[p^\infty],
\end{align*}
where the direct sum is taken over all primes $w$ of $L$ lying above $v$.
\item[\textup{(}ii\textup{)}]For an infinite algebraic extension $\cL/\Q$, define $K_v^1\left(\EC/\cL\right)$ by taking the inductive limit of $K_v^1\left(\EC/L\right)$ over all finite extensions $L/\Q$ contained in $\cL$, and similarly for $J_v^1(\EC/\cL)$.
\item[\textup{(}iii\textup{)}]Let $L$ be an algebraic extension of $\Q$ (possibly infinite) that is contained inside $\Q_\Sigma$, define the \emph{$p^\infty$-fine Selmer group} of $\EC$ over $L$ as
\[
\Sel^0(\EC/L):=\ker\left(H^1(G_{\Sigma}(L),\EC[p^\infty])\longrightarrow\bigoplus_{v\in \Sigma(L)} K^1_v(\EC/L)\right).
\]
and the \emph{$p^\infty$-Selmer group} as
\[
\Sel(\EC/L):=\ker\left(H^1(G_{\Sigma}(L),\EC[p^\infty])\longrightarrow\bigoplus_{v\in \Sigma(L)} J^1_v(\EC/L)\right).
\]
\end{itemize}
\end{definition}

We will be interested in Selmer groups and fine Selmer groups over the cyclotomic $\Zp$-extension of a number field.
Given a number field $K$ and an elliptic curve $\EC/K$, we view the Selmer group over $K_{\cyc}$ as the direct limit of the Selmer groups over the number fields $K_{(n)}$ in $K_{\cyc}$, i.e.
\[
\Sel(\EC/K_{\cyc})=\varinjlim_n \Sel(\EC/K_{(n)}) = \ker\left(H^1(G_{\Sigma}(K_{\cyc}),\EC[p^\infty])\longrightarrow\bigoplus_{v\in \Sigma(K_{\cyc})} J^1_v(\EC/L)\right),
\]
and similarly for fine Selmer groups.
Since all primes are finitely decomposed in the cyclotomic $\Zp$-extension, we henceforth simplify notation and write $\bigoplus_{v\in {\Sigma}(K_{\cyc})} H^1\left(K_{\cyc,v},\EC[p^\infty]\right)$ in place of $\bigoplus_{v\in {\Sigma}(K_{\cyc})} K^1_v(\EC/K_{\cyc})$, and similarly for $\bigoplus_{v\in {\Sigma}(K_{\cyc})} J^1_v(\EC/K_{\cyc})$.
Since $p$ is fixed, we ignore including it in the notation of the (fine) Selmer group.
We remind the reader that both the definitions are independent of ${\Sigma}$ as long as it \emph{contains} the primes above $p$, the primes of bad reduction of $\EC$, and the archimedean primes.

Let $L/\Q$ be an algebraic extension.
The Selmer group sits inside the following short exact sequence
\[
0 \longrightarrow \EC(L) \otimes \Qp/\Zp \longrightarrow \Sel(\EC/L) \longrightarrow \Sha(\EC/L)[p^\infty] \longrightarrow 0.
\]
The left-most object is called the \emph{Mordell--Weil group} of $\EC$ and the right-most object is called the \emph{Shafarevich--Tate group}.
The relationship between the $p^\infty$-fine Selmer group and the $p^\infty$-Selmer group is made clear through the following exact sequence
\[
0 \longrightarrow \Sel^0(\EC/L) \longrightarrow \Sel(\EC/L) \longrightarrow \bigoplus_{v\mid p} \EC(L_v) \otimes \Qp/\Zp.
\]

The Pontryagin dual of $\Sel(\EC/K_{\cyc})$ denoted by $\Sel(\EC/K_{\cyc})^\vee$ is a finitely generated $\Lambda$-module.
When $p$ is a prime of good \textit{ordinary} reduction and $K/\Q$ is abelian, it is known that $\Sel(\EC/K_{\cyc})$ is in fact, $\Lambda$-cotorsion; see \cite[Theorem~17.4]{Kat04}.
We can therefore attach Iwasawa invariants to $\Sel(\EC/K_{\cyc})^\vee$.
Since $\Sel^0(\EC/K_{\cyc})^\vee$ is a quotient of $\Sel(\EC/K_{\cyc})^\vee$, the same assertions hold for the (dual) fine Selmer group too.

For the purposes of our proofs, we are also interested in the \emph{$p^k$-fine Selmer group} for $k\geq 1$ which is defined for algebraic extensions $L/\Q$ as follows
\[
\Sel^0_{\Sigma}(\EC[p^k]/L) = \ker\left( H^1(G_{\Sigma}(L), \EC[p^k]) \longrightarrow \bigoplus_{v\in \Sigma(L)} H^1(L_v, \EC[p^k]) \right).
\]
As indicated in the notation, this \emph{may depend} on the choice of $\Sigma$.
However, if $L$ contains the cyclotomic $\Zp$-extension (as will be in our case) then the choice of $\Sigma$ does not matter (as long as $\Sigma$ contains $p$, the primes of bad reduction of $\EC$, and the archimedean primes).
%Sujatha-Witte

\subsection{\texorpdfstring{$p$-adic $L$-functions}{}}
\label{sec: p-adic L function defn}
Let $\EC/\Q$ be an elliptic curve.
Recall that $\EC$ is modular, i.e there exists a modular form $f_{\EC}$ in $S_2(\Gamma_0(N_{\EC}))$ having the same Fourier-coefficients as $\EC$.
Let $\psi'$ be a finite order Dirichlet character and set $f_{\EC,\psi'}$ to denote the twist of $f_{\EC}$ by $\psi'$.

Suppose that $\psi$ is a non-trivial character on $\Gal(\Q_{\cyc}/\Q)$ of order $p^t$, and $\psi'$ is the induced Dirichlet character.
Assume that $\EC$ has good ordinary reduction at $p$ and let $\mathcal{L}(f_{\EC})\in \Lambda$ be the element defined by B.~Mazur, J.~Tate, and L.~Teitelbaum in \cite[Section 13]{MTT}.

Then we have the following interpolation property \cite[Section 14]{MTT}
\[
\psi(\mathcal{L}(f_{\EC}))=\alpha_p^{-t-1}\frac{p^{t+1}}{G({\psi'}^{-1})}L(f_{\EC, {\psi'}^{-1}},1)/\Omega_{f_{\EC}},
\]
where $\alpha_p$ is the unit root of $X^2-a_p(f_{\EC})X+p$ in $\Qp$, $G(\psi'^{-1})$ is the Gauss--sum and $\Omega_{f_{\EC}}$ is (real) the period associated to $f_{\EC}$.
Recall that $a_p(f_{\EC})$ is the $p$-th coefficient of the Fourier expansion (equivalently it is trace of the Frobenius endomorphism at the prime $p$).

\subsection{Iwasawa Main Conjecture}
The Iwasawa main conjecture relates the analytic object $\mathcal{L}(f_{\EC})$ to the algebraic object $f^{(p)}_{\Sel(\EC/\Q_{\cyc})^\vee}$.
More precisely:

\begin{conj}[Greenberg--Iwasawa Main conjecture]
Let $\EC/\Q$ be an elliptic curve with good ordinary reduction at $p$.
Then there exists the following equality of ideals in $\Lambda$
\[
\langle \mathcal{L}(f_{\EC}) \rangle = \langle f^{(p)}_{\Sel(\EC/\Q_{\cyc})^\vee} \rangle.
\]
\end{conj}

The Greenberg--Iwasawa Main conjecture is known in the following cases in view of the work by C.~Skinner--E.~Urban \cite[Theorem 1]{skinner-urban} and K.~Rubin \cite[Theorem 12.3]{rubin}.

\begin{theorem}
\label{Iwasawa main}
Let $\EC/\Q$ be an elliptic curve of conductor $N_{\EC}$.
Let $p$ be an odd prime of good ordinary reduction.
Assume that either of the following conditions holds
\begin{enumerate}
    \item[\textup{(}1\textup{)}] $\EC$ has complex multiplication by an imaginary quadratic field.
    \item[\textup{(}2\textup{)}] $ \EC$ satisfies the following conditions:
        \begin{enumerate}
            \item[\textup{(}a\textup{)}] the natural representation $\rho$ of $G_\Q$ on $\EC[p]$ is irreducible.
            \item[\textup{(}b\textup{)}] there exists a prime $q\mid N_{\EC}$, $q^2\nmid N_{\EC}$ such that $\rho$ is ramified at $q$.
        \end{enumerate}
    \end{enumerate}
Then the Greenberg--Iwasawa Main conjecture holds in $\Lambda\otimes \Q_p$, i.e., there exist non-negative integers $a$ and $b$ such that 
\[
\langle p^a \mathcal{L}(f_{\EC}) \rangle =\langle p^bf_{\Sel(\EC/\Q_{\cyc})^\vee}^{(p)} \rangle.
\]
In particular, $\lambda(\mathcal{L}(f_\EC))=\lambda_{\alg}(\Sel(\EC/\Q_{\cyc})^\vee)$.
If Condition~\textup{(}1\textup{)} is satisfied, then the Greenberg--Iwasawa Main conjecture is true in $\Lambda$.
\end{theorem}

In what follows we refer to the $\lambda(\mathcal{L}(f_\EC))$ as the \emph{analytic} $\lambda$-invariant and denote it by $\lambda_{\an}(\EC)$.

\subsection{Cusp forms of weight \texorpdfstring{$3/2$}{}}
The (Hecke) congruence subgroup of level $N$ is defined as the following subgroup for a fixed (positive) integer $N$
\[
\Gamma_0(N) = \left\{ \begin{pmatrix} a & b \\ c & d\end{pmatrix} \in \SL_2(\Z) \colon c \equiv 0 \pmod{n}\right\}.
\]
Write $\mathbb{H}$ to denote the upper half plane and let $z\in \mathbb{H}$.
Set $q = e^{2\pi i z}$.
Define the \emph{theta-function} as
\[
\Theta(z):=\sum_{n\in \Z}q^{n^2}.
\]

Let $c$ be any integer and $d$ be an odd integer.
Define the \emph{quadratic residue symbol} as follows
\[
\left(\frac{c}{d}\right) := \begin{cases}
    \text{Jacobi symbol } \left(\frac{c}{d}\right) & \text{ if } d >0\\
    \text{Jacobi symbol } \left(\frac{c}{|d|}\right) & \text{ if } c>0, \ d <0\\
    \text{Jacobi symbol } - \left(\frac{c}{|d|}\right) & \text{ if } c,d <0\\
    1 & \text{ if } c=0, \ d=\pm 1;
\end{cases}
\]
For $d$ odd, further define
\[
\varepsilon_d := \begin{cases}
    1 & \text{ if } d\equiv 1\pmod{4}\\
    i & \text{ if } d\equiv 3\pmod{4}.
\end{cases}
\]
When $\gamma\in \Gamma_0(4)$ and $z\in \mathbb{H}$, define %\cite[p.~148]{koblitz}
\[
j(\gamma,z) : =\left(\frac{c}{d}\right)\varepsilon_d^{-1}\sqrt{cz+d}.
\]
Then, \cite[Theorem on p.~148]{koblitz} asserts that for $z\in \mathbb{H}$,
\[
j(\gamma,z)=\frac{\Theta(\gamma z)}{\Theta(z)}.
\]

We now define a modular form of half-integral weight.
Intuitively, the idea is as follows: for $k$ a positive odd integer, we decide that $\Theta^k$ is a weight 2 modular form of level 4 with trivial character.
Then, we define an arbitrary modular form of weight $k/2$ to be a holomorphic function which transforms under (a congruence subgroup of) $\Gamma_0(4)$ in the same way as $\Theta^k$ does.
More precisely,

\begin{definition}
Let $N$ be a positive integer divisible by $4$.
Let $\chi$ be a character modulo $N$ and $k$ be a positive integer.
A meromorphic function $F$ is a modular form of weight $k/2$, Nebentypus $\chi$ and level $N$ if it satisfies
\[
F(\gamma z)=\chi(d)j(\gamma,z)^k F(z)\quad\forall \gamma=\begin{pmatrix} a&b\\c&d\end{pmatrix}\in \Gamma_0(N).
\]
\end{definition}

Let $\chi$ be a quadratic character of conductor $r$ and for a natural number $t$, write $\chi_t$ to denote the quadratic character corresponding to $\Q(\sqrt{t})/\Q$.
By \cite[Proposition~2.2]{shimura}, we know that 
\[
F_\rho:=\sum_m \chi(m)mq^{tm^2}
\]
is a cusp form of weight $3/2$, level $4r^2t$ and character $\rho = \chi \chi_t\chi_{-1}$.

We write $S_{3/2}(N,\psi)$ to denote the subspace of all cusp-forms of level $N$, nebentypus character $\psi$ and weight $3/2$ that are orthogonal to all cusp forms $F_\rho$ under the Petersson inner product.
We write $S_{3/2}(N)$ to denote $\oplus_\psi S_{3/2}(N,\psi)$, where $\psi$ runs over all characters modulo $N$.
Let $f = \sum a_n(f)q^n$ be a newform of weight $2$ and define
\[
S_{3/2}(N,\chi, f)=\{ F\in S_{3/2}(N,\chi)\mid T(p^2)F=a_p(f) \ \forall p\nmid N\}.
\]

\begin{theorem}[J.-L.~Waldspurger]
Let $g$ be a modular form of odd square-free conductor.
Let $G\in S_{3/2}(4N_{\EC},\chi,g)$ be a modular form such that $G$ is the Shimura lift of $g$. 
Let $n_1$ and $n_2$ be square-free natural numbers such that $\frac{n_1}{n_2}\in (\Q_q^*)^{2}$ for all $q\mid N$.
Then 
\[
a_{n_1}^2(G) \chi\left(\frac{n_2}{n_1}\right) \sqrt{n_2}  L(g_{\chi^{-1}\chi_{-n_2}},1)=a_{n_2}^2(G)\sqrt{n_1}L(g_{\chi^{-1}\chi_{-n_1}},1)
\]
\end{theorem}

\begin{proof}
See \cite[Corollary~2]{waldspurger}.
\end{proof}

\section{Approach I: Analytic approach}
\label{sec: approach 1}

We would like to apply {the above} result to elliptic curves $\EC/\Q$ and the associated cusp forms $f_{\EC}$.
However, Waldspurger's result only applies to $f_{\EC}$ and not to the twist $f_{\EC}\otimes \psi$, where $\psi$ is a finite order character.

\begin{theorem}
\label{thm.analytic twist}
Let $\EC/\Q$ be an elliptic curve of square-free odd conductor {$N_{\EC}$}.
Let $F\in S_{3/2}(4N_{{\EC}},1,f_{\EC})$ be a modular form of weight $3/2$ corresponding to $f_{\EC}$ under the Shimura lift.
Let $n_1$, $n_2$ be square-free integers coprime to $p$ such that $\frac{n_1}{n_2}\in (\Q_q^\times)^2$ for all $q\mid 4N_{\EC}p$.
Assume that $\chi_{-n_1}$, $\chi_{-n_2}$ have the same sign and $a_{n_2}(F)$ is a $p$-adic unit.
{If} $\lambda_{\an}(\EC^{(-n_1)}/\Q_{\cyc})=\mu_{\an}(\EC^{(-n_1)}/\Q_{\cyc})=0$, then the same holds for $\EC^{(-n_2)}$.
\end{theorem}

\begin{proof}
Note that for both $i=1,2$ we know that $X^2-a_p\chi_{-n_i}(p)(X)+p$ have the same unit root as $n_1/n_2$ is a square modulo $p$.
Furthermore $\EC^{(-n_2)}$ and $\EC^{(-n_1)}$ have periods that only differ by a $p$-adic unit as $\chi_{-n_1}$ and $\chi_{-n_1}$ have the same sign {by assumption}; {see} \cite[page 40]{MTT}.
{Let us} denote $\Omega_{-n_1}/\Omega_{-n_2}$ by $u$.
Using Waldspurger's theorem we obtain
\begin{align*}
a^{2}_{n_1}(F)\mathbf{1}(\mathcal{L}(f_{\EC,\chi_{-n_2}}))& = a^{2}_{n_1}(F)\alpha_p^{-1}\left(1-\frac{1}{\alpha_p}\right)^2\frac{L(f_{\EC,\chi_{-n_2}},1)}{\Omega_{-n_2}^{-1}}\\
&= a^{2}_{n_1}(F)\alpha_p^{-1}\left(1-\frac{1}{\alpha_p}\right)^2\frac{L(f_{\EC,\chi_{-n_2}},1)}{\Omega_{-n_1}^{-1}}\\
&=a_{n_2}^2(F)\alpha_p^{-1}\left(1-\frac{1}{\alpha_p}\right)^2\sqrt{n_1/n_2}\frac{L(f_{\EC,\chi_{-n_2}},1)}{u\Omega_{-n_2}^{-1}}\\
&=u^{-1}a_{n_2}^2(F)\sqrt{n_1/n_2}\mathbf{1}(\mathcal{L}(f_{\EC,\chi_{-n_1}})).
\end{align*}
Here we denote by $\mathbf{1}(\cdot)$ the evaluation of an element in $\Lambda$ at the trivial character. 
By assumption the right hand side of this equation is a $p$-adic unit. As the Fourier coefficients of $F$ are $p$-adic units we obtain that $\mathbf{1}(\mathcal{L}(f_{\EC,\chi_{-n_2}}))$ is a $p$-adic unit. In particular, its Iwasawa invariants vanish.
\end{proof}

As an immediate corollary we obtain

\begin{corollary}
Suppose that in addition to the assumptions of Theorem~\ref{thm.analytic twist} the Iwasawa Main Conjecture holds for both twists.
Then the algebraic Iwasawa invariants vanish, as well.
\end{corollary}

The critical point now to understand is when are the Fourier coefficients of $F$ not divisible by $p$.
K.~Prasanna \cite{Pra09, Pra10} proved an asymptotic formula on these Fourier coefficients involving special values of twisted $L$-series.
Using Prasanna's result for the non-vanishing of the Fourier coefficient is circular in our application.
Therefore, instead we follow the approach of J.~Antoniades, M.~Bungert, and G.~Frey \cite[Section~4]{frey-all}.
Let $\EC/\Q$ be an elliptic curve of prime conductor $\ell$ satisfying $N_{\EC} = \ell \equiv 3\mod 4$.
Assume that $p$ does not divide the class number of $\Q(\sqrt{-\ell})$, that $\EC(\Q)$ contains a point of order $p$, and that $p$ is not a congruence prime for $f_{\EC}$, i.e the space of cups forms congruent to $f_{\EC}$  modulo $p$ is $1$-dimensional.
Note that the existence of a $p$-torsion point implies that $p\in \{3,5,7\}$. 

\begin{theorem}
\label{frey}
Suppose that the assumption introduced in the previous paragraph hold. 
Let $n_1$ be a square-free integer such that $p$ does not divide $a_{n_1}L(\EC^{(-n_1)},1)$.
Further suppose that $n_2$ is a square-free integer such that $n_1/n_2\in (\Q_q^\times)^2$ for all $q\mid 4\ell$ and $n_2$ is coprime to $\ell$.
Then $p\mid a_{n_2}$ if and only if $p\mid h(-n_2)$,
where $h(-n_2)$ denotes the class number of $\Q(\sqrt{-n_2})$.
\end{theorem}

\begin{proof}
This is \cite[Proposition~4.8]{frey-all}.    
\end{proof}

Combining Theorems \ref{thm.analytic twist} and \ref{frey} we obtain:

\begin{theorem}
\label{central-thm-analytic}
Keep the assumptions on $\EC/\Q$, $n_1$, and $p$ as stated in Theorem \ref{frey} and the paragraph preceding it.
Assume that the Iwasawa invariants of $\EC^{(-n_1)}$ vanish. Then for all square-free $n_2$ coprime to $\ell$ such that $n_1/n_2\in (\Q_q^\times)^\times$ and such that the class group of $\Q(\sqrt{-n_2})$ is not divisible by $p$, the Iwasawa invariants vanish for $\EC^{(-n_2)}$.
\end{theorem}

\begin{proof}
We only need to check that $\chi_{-n_1}$ and $\chi_{-n_2}$ have the same sign.
Observe that $n_1\equiv n_2\mod 4$ as $n_1/n_2\in \Q_2^2$.
In particular, $n_1 n_2\equiv 1\mod 4$.
Let $m>0$ be a prime congruent to $-1$ modulo $\textup{cond}(\chi_{-n_1})\textup{cond}(\chi_{-n_2})$.
We obtain
\[
\chi_{-n_1}\chi_{-n_2}(-1)=\left(\frac{n_1n_2}{m}\right)=\left(\frac{m}{n_1}\right)\left(\frac{m}{n_2}\right)=\left(\frac{-1}{n_1}\right)\left(\frac{-1}{n_2}\right)=1.
\]
Thus, the characters $\chi_{-n_1}$ and $\chi_{-n_2}$ have the same sign and we can indeed apply Theorem \ref{thm.analytic twist}.
\end{proof}

\section{Approach II: via explicit splitting conditions}
\label{sec: approach 2}
Let $\EC/\Q$ be an elliptic curve.
Let $L/\Q$ be an algebraic extension and let $\Sigma(L)$ denote the set of all primes (in $L$) above the prime divisors of the conductor of $\EC$, the primes above $p$, and the primes ramified in $L/\Q$.

\begin{assumption}\label{ass}Let $K$ be a quadratic field such that the following conditions are satisfied
    \begin{enumerate}
    \item[\textup{(}i\textup{)}] $\EC(K_v)[p]=\{0\}$ for all $v\in \Sigma(K)$.
    \item[\textup{(}ii\textup{)}] $\EC$ has good ordinary reduction at $p$ odd.
    \item[\textup{(}iii\textup{)}] $p$ splits in $K$.
    \item [\textup{(}iv\textup{)}] $\Sel^0(\EC/K)= \{0\}$.
   \item[\textup{(}v\textup{)}]
    Assume that $\EC/K$ is of rank $2$. 
\end{enumerate}
\end{assumption}
If not specified otherwise we will assume that Assumption \ref{ass} is valid for the rest of this section.

\begin{remark}
\label{rem:rank}
In view of conditions (i)--(iii) for $\fp\mid p$, the hypothesis $\EC(K_\fp)[p] = \EC(\Q_p)[p]=\{0\}$ means that $p$ is a non-anomalous prime.
Furthermore, (i) implies that $\EC(K_{\cyc,v})[p] = 0$ for all $v\in \Sigma(K_{\cyc})$ by \cite[Proposition~5.1]{HM99}.
This is straightforward to see when $v$ is a prime of good reduction or non-split multiplicative reduction or additive reduction for $\EC$.
Note that when $v\in \Sigma(K)$ is a prime of split multiplicative reduction and $\mu_p\subseteq K_v$, then $\EC(K_v)[p]$ is non-trivial.
%\DK{I don't think this is really true.
%What if $\Q_{\ell}$ already contains $\mu_p$?
%This can happen, right? if $p\mid \ell-1$ or something like that and $\ell$ is a prime of split multiplicative reduction?
%The split multiplicative reduction was the problem case because of the Tate curve issue.}
Condition (iv) implies that $\EC/K$ has Mordell--Weil rank $\leq 2$ and $\EC/\Q$ has rank at most $1$; \cite[Corollary~7.3]{wuthrich2007fine}.
We remind the reader of the expectation that 100\% of the elliptic curves over $\Q$ have rank at most 1.
We show that under finiteness assumptions on the Shafarevich--Tate group, the elliptic curve $\EC$ and its twist $\EC^{K}$ have the same rank over $\Q$.

%In view of (i), it is clear that $\EC(K)[p] = \EC(\Q)[p] = \{0\}$.
%If both $\EC$ and $\EC^K$ have Mordell--Weil rank equal to 0 over $\Q$, then (v) is satisfied precisely when $\Sha(\EC/\Q)[p^\infty]$ and $\Sha(\EC^K/\Q)[p^\infty]$ are \emph{both} non-trivial.
%In particular, the $p$-primary parts of the fine Shafarevich--Tate groups (over $\Q$) for both $\EC$ and $\EC^K$ are non-trivial; see \cite[Theorem~3.4]{wuthrich2007fine}.
%This forces (iv) to not be satisfied.
We discuss the situation when the Mordell--Weil rank of $\EC/K$ is equal to 0 at the end in Section~\ref{sec: rank 0 discussion}. 
\end{remark}

\subsection{Selmer groups of quadratic twist families of elliptic curves}
In this section we provide a proof of Theorem~\ref{thm B}.
The idea is to show that under Assumption \ref{ass} the global and local cohomology groups over $K$ are isomorphic to the ones over $\Q$ twisted by $\Z_p[G]$, where $G=\Gal(K/\Q)$.
In particular, the (dual) Selmer group over $K_{\cyc}$ is isomorphic to the one over $\Q$ twisted by $\Z_p[G]$.
To do this, we need a refined understanding of the structure of the BDP-Selmer group over $K_{\cyc}$; see Proposition~\ref{selbdp}.
In the final step, we decompose $\Z_p[G]$ via its two central idempotents allowing us to identify the two Selmer groups in question.

We make a simple observation that will be useful for our arguments:

\begin{remark}
\label{vanishing}
In view of \cite[Proposition~2]{greenberg}, the Galois cohomology group $H^1(K_{(n),w},\EC[p^\infty])=0$ for all $w\in \Sigma{(K_{(n)})\setminus \Sigma_p(K_{(n)})}$. \end{remark}

\begin{definition}
Fix a prime $\fp$ above $p$ in $K$.
Fix a norm-coherent sequence of primes $v_n$ above $\fp$.
Define
\begin{align*}
\SelBDP(\EC/K_{(n)}) &=\ker\left(H^1({G_{\Sigma}}(K_{(n)}),\EC[p^\infty])\longrightarrow \prod_{w\in \Sigma(K_{(n)})\setminus \{v_n\}}H^1(K_{(n),w},\EC[p^\infty])\right) \\
\Sel^{0,\Gr}(\EC/K_{(n)})&=\ker\left(H^1({G_{\Sigma}}(K_{(n)}),\EC[p^\infty])\longrightarrow \prod_{w\in \Sigma(K_{(n)})\setminus \{v_n\}}H^1(K_{(n),w},\EC[p^\infty])\times \frac{H^1(K_{(n),v_n},\EC[p^\infty])}{\EC(K_{(n),v_n})\otimes \Q_p/\Z_p}\right).
\end{align*}
Further, define $\SelBDP(T_p\EC/K_{(n)})$ similar to $\SelBDP(\EC/K_{(n)})$ with $\EC[p^\infty]$ replaced by $T_p\EC$.
\end{definition}

The BDP Selmer group satisfies the following control theorem which was proven in \cite[Theorem~4.1]{lei-lim-m}.
\begin{theorem}[Control Theorem]
\label{control-bdp-selmer}
Under Assumption~\ref{ass}\textup{(}i\textup{)} the natural homomorphism
\[
\SelBDP(\EC/K_{(n)})\longrightarrow \SelBDP(\EC/K_{\cyc})^{\Gamma_n}
\]
is an isomorphism.
\end{theorem}

\subsubsection{Structure of \texorpdfstring{$\SelBDP(\EC/K_{\cyc})^\vee$}{} as a \texorpdfstring{$\Lambda$}{}-module}

\begin{lemma}
Let $\fp\mid p$ be a prime in $K$.
Then $H^1(K_{\fp},\EC[p^\infty])\cong (\Qp/\Z_p)^2$.
\end{lemma}

\begin{proof}
We write {$v_\infty$} to denote the unique prime in $K_{\cyc}$ above $\fp$ in $K$.
%\DK{\sf \textbf{[DK: I should change these $\fp$ to $v_\infty$ over $K_{\cyc}$? and keep $\fp\mid p$ over $K$? But $\Gamma_{\fp}$ stays as is? change!!]}}
By the inflation-restriction exact sequence we have
\[
0 \longrightarrow H^1(\Gamma_{{v_\infty}}, \EC(K_{\cyc,{v_\infty}})[p^\infty]) \longrightarrow H^1(K_{\fp},\EC[p^\infty]) \longrightarrow H^1(K_{\cyc,{v_\infty}}, \EC[p^\infty])^{\Gamma_{\fp}} \longrightarrow 0.
\]
The surjectivity on the right hand side follows from the fact that $\Gamma_{\fp} \simeq \Gamma$ has $p$-cohomological dimension 1.
The first term is trivial by our assumption that $\EC(K_\fp)[p]=\{0\}$.
Recall that \cite[Corollary~2]{greenberg} asserts that $H^1(K_{\cyc,{v_\infty}}, \EC[p^\infty])^\vee \simeq \Lambda^2$.
The result now follows.
\end{proof}

We are going to need the following lemma which asserts that $p^k$-fine Selmer group $\Sel^0(\EC[p^k]/K)$ which a priori depends on the set $\Sigma(K)$, does not depend on this choice under our additional assumptions.

\begin{lemma}
\label{isom:different-selmer}
Suppose that Assumption~\ref{ass}\textup{(}i\textup{)} is satisfied\footnote{we do not insist that the conditions (ii)--{(v)} hold.}.
Let $k$ be a positive integer.
Then
\[
\Sel^0(\EC/K)[p^k]\cong \Sel^0(\EC[p^k]/K).
\]
\end{lemma}

\begin{proof}
Consider the following commutative diagram:
\[
\begin{tikzcd}
        0\arrow[r] &\Sel^0(\EC[p^k]/K) \arrow[r]\arrow[d]&H^1(G_\Sigma(K),\EC[p^k])\arrow[r]\arrow[d,"h"]& \prod_{v\in \Sigma(K)}H^1(K_{v},\EC[p^k])\arrow[d,"g"]\\
        0\arrow[r] &\Sel^0(\EC/K)[p^k] \arrow[r]&H^1(G_\Sigma(K),\EC[p^\infty])[p^k]\arrow[r]& \prod_{v\in \Sigma(K)}H^1(K_{v},\EC[p^\infty])[p^k]
\end{tikzcd}
\]
The maps $h,g$ are the usual maps that arise in the definition of the Kummer sequence.
Recall that 
\begin{align*}
\ker(h) & = H^0(G_\Sigma(K),\EC[p^\infty])/p^k\\
\ker(g) & =\prod_{v\in \Sigma(K)}H^0(K_v,\EC[p^\infty])/p^k;
\end{align*}
but both of the cohomology groups are trivial by Assumption~\ref{ass}(i).
Furthermore, $h$ is surjective by definition.
The claim now follows from the snake lemma.
\end{proof}

\begin{remark}
For the purposes of this proof we could have worked with a slightly smaller set $\Sigma'(K) \subset \Sigma(K)$ which excluded the ramified primes in $K/\Q$ and required that $\EC(K_v)[p]=\{0\}$ for all $v\in \Sigma'(K)$.
However, in the future we will want to prove a similar result for the quadratic twist $\EC^K$.
In that situation, we will see that the primes of bad reduction of $\EC^K$ include the ramified primes of $K/\Q$.
In order to be able to provide a uniform proof we have required the full force of Assumption~\ref{ass}(i).
\end{remark}

\begin{definition}
Let $G=\Gal(K/\Q)=\langle \iota \rangle$.
Let $M$ be a $\Z_p[G]$-module.
Write $M^{1\pm \iota}$ to denote the maximal submodule of $M$ on which $\iota$ acts via $\pm 1$.
\end{definition}

As $p$ is odd, we know that $2$ is a unit in $\Z_p$.
The idempotent $\frac{1\pm \iota}{2}$ is an element in $\Z_p[G]$.
In particular, we have a decomposition $M\cong M^{1+\iota}\oplus M^{1-\iota}$ as $\Z_p[G]$-modules.
If $M$ is a $\Lambda[G]$-module, th{e aforementioned isomorphism} is even a decomposition of $\Lambda[G]$-modules.

\begin{lemma}
\label{global}
The following isomorphism holds true under Assumption~\ref{ass}(i) and (iv)
\[
H^1(G_\Sigma(K_{\cyc}),\EC[p^\infty])^\vee\cong \Lambda[G].
\]
\end{lemma}

\begin{proof}
By local Tate-duality and using Assumption~\ref{ass}(i) we have that
\[
H^2(K_v,\EC[p^k])=0 \text{ for all } v\in \Sigma(K).
\]
In view of Assumption~\ref{ass}(iv) and using Lemma~\ref{isom:different-selmer} we also have that $\Sel^0(\EC[p^k]/K)=0$.
Then by global Tate-duality we obtain $H^2(G_\Sigma(K),\EC[p^k])=0$.

Now, consider the sequence
\[
0 \longrightarrow \EC[p^k] \longrightarrow \EC[p^\infty] \xrightarrow{\times p^k} \EC[p^\infty] \longrightarrow 0.
\]
Taking the long exact sequence of cohomology, implies that $H^1(G_\Sigma(K), \EC[p^\infty])$ is divisible.
The Euler--Poincare characteristic formula \cite[Equation~27]{greenberg} now gives
\[
\abs{H^1(G_\Sigma(K),\EC[p^k])}=p^{2k}.
\]
Thus, $H^1(G_\Sigma(K),\EC[p^\infty])$ is divisible of rank $2$.
Note that $K/\Q$ is unramified outside $\Sigma(K)$.

By the inflation-restriction exact sequence we have
\[
0 \longrightarrow H^1(\Gamma, \EC(K_{\cyc})[p^\infty]) \longrightarrow H^1(G_{\Sigma}(K),\EC[p^\infty]) \longrightarrow H^1(G_{\Sigma}(K_{\cyc}), \EC[p^\infty])^{\Gamma} \longrightarrow 0.
\]
The first term is trivial by the assumption that $\EC(K_v)[p]$ is trivial for all $v\in \Sigma(K)$ because $\EC(K)[p^\infty]$ must be trivial under our assumptions.
Since $H^1(G_\Sigma(K),\EC[p^\infty])$ is divisible of $\Z_p$-corank $2$ we obtain that $H^1(H_\Sigma(K_{\cyc}),\EC[p^\infty])^{\Gamma}$ is divisible of $\Z_p$-corank $2$ as well.
Now \cite[Proposition~3]{greenberg} implies that the $\Lambda$-corank of $H^1(G_\Sigma(K_{\cyc}),\EC[p^\infty])$ is 2.
Thus, (as $\Lambda$-module)
\[
H^1(G_\Sigma(K_{\cyc}),\EC[p^\infty])^\vee\cong \Lambda^2.
\]
Since $H^1(G_\Sigma(K_{\cyc}),\EC[p^\infty])^{1\pm \iota}$ has $\Lambda$-corank at least 1, we obtain $H^1(G_\Sigma(K_{\cyc}),\EC[p^\infty])^\vee\cong \Lambda[G]$.
\end{proof}

Write $\Sigma'(K_{(n)})$ to denote the set of primes $\Sigma(K_{(n)})\setminus\{v_n\}$.
The Cassels--Poitou--Tate exact sequence gives
\begin{multline*}
0 \longrightarrow \SelBDP(\EC/K_{(n)})\to H^1({G_{\Sigma}}(K_{(n)}),\EC[p^\infty]) \longrightarrow \prod_{w\in \Sigma'(K_{(n)})}H^1(K_{(n),w},\EC[p^\infty]) \longrightarrow \\
\SelBDP(T_p\EC/K_{(n)})^\vee \longrightarrow H^2({G_{\Sigma}}(K_{(n)}),\EC[p^\infty])=0.
\end{multline*}
Taking limits we see that 
\[
\Lambda\textup{-corank}(\SelBDP(\EC/K_{\cyc}))=\Lambda\textup{-rank}(\varprojlim_ n\SelBDP(T_p\EC/K_{(n)})).
\]

\begin{lemma}
\label{E/K rank 2 then twisted SelBDP corank 1}
Assume that $\EC/K$ is of rank $2$.
Let $\kappa$ be any power of the cyclotomic character.
Then $\SelBDP(\EC[p^\infty]\otimes \kappa/K)$ is of corank $1$.
\end{lemma}

\begin{proof}
Under our running assumptions we have $H^1(G_\Sigma(K),\EC[p^\infty])=\EC(K)\otimes \Q_p/\Z_p$.
Let $P$ be a point such that $p^mP\in \EC(K)\setminus p\EC(K)$ and let $\phi_P$ be the cocyle given by $\phi_p(\sigma)=\sigma P-P$.
Let $\kappa$ be a power of the cyclotomic character.
Define $\phi_p^\kappa(\sigma)=\kappa(\sigma)\sigma P-P$.
By definition,
\[
\phi_p^\kappa(\sigma\tau)=\kappa(\sigma\tau)\sigma\tau P-P=\kappa(\sigma)\sigma(\kappa(\tau)\tau P-P)+(\kappa(\sigma) \sigma P-P),
\]
which shows that $\phi_p^\kappa$ is again a coclye. Applying the twist to a coboundary gives a coboundary. 
Thus, we obtain a map
\[
H^1(G_\Sigma(K),\EC[p^\infty]) \longrightarrow H^1(G_\Sigma(K),\EC[p^\infty]\otimes \kappa).
\]
We claim that this map is injective.
Indeed, assume that $\phi_P^\kappa=\phi_Q^\kappa$ for points that define different cocycles $\phi_P$ and $\phi_Q$.
Then
\[
\kappa(\sigma)\sigma(P-Q)=P-Q
\]
for all $\sigma\in G_\Sigma(K)$.
Take $\sigma \in G_\Sigma(K_\infty)$ then $\kappa(\sigma)=1$ and we see that that $\sigma(P-Q)=P-Q$.
In particular, $P-Q\in \EC(K_\infty)$.
Furthermore, there exists an $m$ such that $p^m(P-Q)\in \EC(K)$.
Injectivity of the map $\EC(K)\otimes \Q_p/\Z_p \to \EC(K_\infty)\otimes \Q_p/\Z_p$ implies that $P-Q\in \EC(K)$.
But then $\phi_P=\phi_Q$.

Now, let $P_1$ and $P_2$ be generators of $\EC(K)$.
There exists an element $c\in \Z_p^\times$ such that $P_1=cP_2$ as points in $\EC(K_{\mathfrak{p}})$.
Let $m$ be a non-negative integer and $P$ a point such that $p^m P=P_1$.
Let $Q$ be a point such that $p^m Q=P_2$.
Without loss of generality we can assume that $P-cQ=0$ in $\EC(\overline{K_{\mathfrak{p}}})$.
Thus, $\phi_P-c\phi_Q\in \SelBDP(\EC/K)$.
It follows that for all $\sigma \in G_{K_{\mathfrak{p}}}$, we have
\[
\kappa(\sigma) \sigma P-P=\kappa(\sigma)\sigma cQ-cQ.
\]
Thus, $\phi_p^\kappa-c\phi_Q^\kappa\in \SelBDP(\EC[p^\infty]\otimes \kappa/K)$.
As $\phi_P^\kappa-c\Phi_Q^\kappa$ is an element of order $p^m$ in $H^1(G_\Sigma(K), \EC[p^\infty]\otimes \kappa)$, we have shown that $\SelBDP(\EC[p^\infty])\otimes \kappa$ contains a cyclic subgroup order $p^m$ for any $m$.
Thus, it contains a divisible subgroup of corank $1$.
  \end{proof}

\begin{proposition}
\label{selbdp}
Suppose that $\EC/K$ has Mordell--Weil rank 2.
Then $\SelBDP(\EC/K_\infty)$ is not $\Lambda$-cotorsion.
If in addition assumption \ref{ass} (i)-(iv) hold, then $\SelBDP(\EC/K_\infty)$ is cofree of rank $1$.
\end{proposition}

\begin{proof}
If $\SelBDP(\EC/K_\infty)$ is cotorsion we can find a power $\kappa$ of the cyclotomic character such that $\EC[p^\infty]\otimes \kappa$ does not contain any point fixed by $G_{K_v}$ for all $v\in \Sigma$ and such that $\Sel(\EC/K_\infty)\otimes \kappa$ has a characteristic ideal coprime to all $\omega_n(T)$.
It follows that 
\[
(\SelBDP(\EC/K_\infty)\otimes \kappa)[T]=\SelBDP(\EC[p^\infty]\otimes \kappa/K_\infty)[T]=\SelBDP(\EC[p^\infty]\otimes \kappa{/K})
\]
is finite, contradicting Lemma~\ref{E/K rank 2 then twisted SelBDP corank 1}.

As $\Sel^0(\EC/K_\infty)=0$, it follows that there is a natural embedding
\[\SelBDP(\EC/K_\infty)\oplus \iota (\SelBDP(\EC/K_\infty))\hookrightarrow H^1(G_\Sigma(K_\infty,\EC[p^\infty])\cong (\Lambda^2)^\vee.\]
It follows that $\SelBDP(\EC/K_\infty)$ is of $\Lambda$-corank $1$ and does not contain a non-trivial torsion submodule. To show that it is indeed cofree it suffices to show that $\SelBDP(\EC/K)$ is of isomorphic to $\Q_p/\Z_p$. The control theorem implies that $\SelBDP(\EC/K)$ is at least of corank $1$. On the other hand we have an embedding
\[\SelBDP(\EC/K)\oplus \iota(\SelBDP(\EC/K))\hookrightarrow H^1(G_\Sigma(K),\EC[p^\infty])\cong (\Q_p/\Z_p)^2.\]
This is only possible if $\SelBDP(\EC/K)\cong \Q_p/\Z_p$. 
\end{proof}

\subsubsection{Proof of Main Theorem}
{Let $G = \Gal(K/\Q)$.
Our strategy is to first show that under our assumptions, we have
\[
\Sel(\EC/\Q_{\cyc})^\vee\otimes \Z_p[G]\cong \Sel(\EC/K_{\cyc})^\vee.
\]
As will be clarified in the course of the proof, to prove the above isomorphism we need to relate the global cohomology group over $K_{\cyc}$ with that over $\Q_{\cyc}$ and similarly for the local cohomology groups.
Even though it might be possible to prove this directly, our proof uses the structure of BDP-Selmer groups.
Once this isomorphism is established, we decompose $\Z_p[G]$ via its two central idempotents to prove that the Selmer groups of $\EC$ and $\EC^{K}$ are isomorphic over the cyclotomic $\Z_p$-extension of $\Q$.
In particular, under reasonable assumptions on the Shafarevich--Tate group the elliptic curve and its twist must have the same rank over $\Q$.}

{Write $\Sigma'(K_{\cyc})$ to denote the set of primes $\Sigma(K_{\cyc})\setminus\{v_{\infty}\}$.}
Let
\[
\alpha \colon H^1(G_\Sigma(K_{\cyc}),\EC[p^\infty])\longrightarrow \prod_{w\in \Sigma'(K_{\cyc})}H^1(K_{\cyc,w},\EC[p^\infty]).
\]
We have a natural embedding $\Image(\alpha)^\vee\to \Lambda^2=H^1(G_\Sigma(K_{\cyc}),\EC[p^\infty])^\vee$ and a surjection 
\[
\prod_{w\in \Sigma'(K_{\cyc})}H^1(K_{\cyc,w},\EC[p^\infty])^\vee=\Lambda^2\longrightarrow \Image(\alpha)^\vee.
\]

As $\SelBDP(\EC/K_{\cyc})^\vee\cong \Lambda$, the control theorem for $\SelBDP(\EC/K_{\cyc})$ implies that we have a decomposition
\begin{equation}
\label{eq 4.1}
H^1({G_{\Sigma}}(K_{\cyc}),\EC[p^\infty])^\vee\cong \SelBDP(\EC/K_{\cyc})^\vee\oplus \Lambda
\end{equation}
as $\Lambda$-modules we define $A=\SelBDP(\EC/K_\infty)^\vee$.

It follows that $\Image(\alpha)^\vee\cong \Lambda$. We choose $C^\vee=\Image(\alpha)$ and obtain a decomposition
\[
H^1(K_{\cyc,\iota v_{{\infty}}},\EC[p^\infty])^\vee =C^\vee\oplus B^\vee
\]
for some $\Lambda$-module $B^\vee$.

Note that there is a natural isomorphism $A\to C$ induced by the global to local map. 
We have a natural embedding of $B$ and $C$ in $\prod_{w\in \Sigma_p(K_{\cyc})}H^1(K_{\cyc,w},\EC[p^\infty])$.
In the following we interpret
\[
H^1(\Q_{\cyc,p},\EC[p^\infty])\otimes \Z_p[G] \cong C\otimes \Z_p[G]\oplus B\otimes \Z_p[G].
\]
Consider the natural projection
\[
\begin{tikzcd}
    H^1({G_{\Sigma}}(K_{\cyc}),\EC[p^\infty])\arrow[r, "\operatorname{proj}"]\arrow[d,dash,shift left=.1em] \arrow[d,dash,shift right=.1em]&H^1({G_{\Sigma}}(\Q_{\cyc}),\EC[p^\infty])\arrow[d,dash,shift left=.1em] \arrow[d,dash,shift right=.1em]\\
    H^1({G_{\Sigma}}(K_{\cyc}),\EC[p^\infty])^{1+\iota}\oplus H^1({G_{\Sigma}}(K_{\cyc}),\EC[p^\infty])^{1-\iota} \arrow[r] & H^1({G_{\Sigma}}(K_{\cyc}),\EC[p^\infty])^{1+\iota}
\end{tikzcd}
\]
Upon restricting this projection to $A$ we obtain a surjection
\[
\operatorname{proj}(A) \simeq H^1({G_{\Sigma}}(\Q_{\cyc}),\EC[p^\infty]).
\]
Indeed, $\ker(\operatorname{proj})\cap A$ is a $\Lambda[G]$-module and therefore contained in $\Sel(\EC/K_{\cyc})$ which is cotorsion. Thus, the cokernel needs to be $\Lambda$-cotorsion.
As $H^1(G_\Sigma(\Q_{\cyc}),\EC[p^\infty])$-is $\Lambda$-divisible, the projection must also be surjective.
The same holds fro the projection top $H^1(G_\Sigma(\Q),\EC^k[p^\infty])$. As $A\cap \iota(A)$ is cotorsion we see that $A+\iota(A)=H^1(G_\Sigma(K_\infty),\EC[p^\infty])$ and as $A$ is $\Lambda$-cofree this is indeed a direct sum.
%image of Sel^{BDP} under this projection is the whole thing
We note the following isomorphism
\[
\iota(A)\otimes \Z_p[G]\cong H^1(G_{\Sigma}(K_{\cyc}),\EC[p^\infty]).
\]
Note furthermore that the global to local map induces an isomorphism
\[
\iota(A)\longrightarrow C.
\]
We can now rewrite \eqref{eq 4.1} as
\[
H^1({G_{\Sigma}}(K_{\cyc}),\EC[p^\infty]) \simeq A\oplus \iota (A).
\]
We therefore obtain a commutative diagram
\[
\begin{tikzcd}
    H^1({G_{\Sigma}}(K_{\cyc}),\EC[p^\infty])\arrow[r] \arrow[d,dash,shift left=.1em] \arrow[d,dash,shift right=.1em] & \prod_{v\in \Sigma_p(K_{\cyc})}H^1(K_{\cyc,v},\EC[p^\infty])\arrow[d,dash,shift left=.1em] \arrow[d,dash,shift right=.1em]\\
    A\oplus \iota (A)\arrow[r]\arrow[d, phantom, "\cong" {anchor=south, rotate=270}, shift right=.5em] & C\oplus B\oplus \iota(C)\oplus \iota (B)\arrow[d, phantom, "\cong" {anchor=south, rotate=270}, shift right=.5em]\\
    (\iota(A))\otimes \Z_p[G]\arrow[r]\arrow[d, phantom, "\cong" {anchor=south, rotate=270}, shift right=.5em] & C\otimes \Z_p[G]\oplus B\otimes \Z_p[G]\arrow[d, phantom, "\cong" {anchor=south, rotate=270}, shift right=.5em]\\
    H^1({G_{\Sigma}}(\Q_{\cyc}),\EC[p^\infty])\otimes \Z_p[G]\arrow[r] & H^1(\Q_{\cyc,p},\EC[p^\infty])\otimes \Z_p[G]
\end{tikzcd}
\]

\begin{proposition}
\label{central-argument}
There is a natural isomorphism
\[
\Sel(\EC/\Q_{\cyc})^\vee\otimes \Z_p[G]\cong \Sel(\EC/K_{\cyc})^\vee.
\]
\end{proposition}

\begin{proof}
Upon dualizing the defining exact sequence of $\Sel(\EC/K_{\cyc})$ we obtain
\[
0 \longrightarrow \prod_{v\in \Sigma(K_{\cyc})}\left(H^1(K_{\cyc,v},\EC[p^\infty])/\EC(K_{\cyc,v})\otimes \Qp/\Z_p\right)^\vee \longrightarrow H^1(G_\Sigma(K_{\cyc}),\EC[p^\infty])^\vee\longrightarrow \Sel(\EC/K_{\cyc})^\vee\longrightarrow 0,
\]
where we have used \cite[Proposition~2.1]{greenberg-vatsal} and the fact that $\Sel(\EC/K_{\cyc})^\vee$ is $\Lambda$-torsion.
A similar short exact sequence is true over $\Q_{\cyc}$, as well.

From the above discussion, the middle term is isomorphic to $H^1(G_\Sigma(\Q_{\cyc}),\EC[p^\infty])^\vee \otimes \Z_p[G]$.
In particular, it does not contain any $\Lambda$-torsion.
Recall that \cite[Proposition~2]{greenberg} asserts that $H^1(K_{\cyc,w},\EC[p^\infty])^\vee=0$ for every place $w\nmid p$; see Remark~\ref{vanishing}.
This means we can replace the product over $\Sigma(K_{\cyc})$ with product over $\Sigma_p(K_{\cyc})$ in the above short exact sequence.

Now we study the local cohomology groups.
As $p$ splits in $K$ we get a similar isomorphism to the global one.
More precisely,
\[
\prod_{v\in \Sigma_p(K_{\cyc})}\left(H^1(K_{\cyc,v},\EC[p^\infty])/\EC(K_{\cyc,v})\otimes \Qp/\Z_p\right)^\vee \simeq \left(H^1(\Q_{\cyc,p},\EC[p^\infty])/\EC(\Q_{\cyc,p})\otimes \Q_p/\Z_p)\right)^\vee\otimes \Z_p[G].
\]
Using the analysis right before this proposition we obtain a commutative diagram
    \[\begin{tikzcd}[font=\tiny, column sep=1em, row sep=1em]
        0\arrow[r]&\displaystyle\prod_{v\in \Sigma_p(K_{\cyc})}\left(H^1(K_{\cyc,v},\EC[p^\infty])/\EC(K_{\cyc,v})\otimes \Qp/\Z_p\right)^\vee\arrow[r]\arrow[d]&H^1(G_\Sigma(K_{\cyc}),\EC[p^\infty])^\vee\arrow[d]\arrow[r]&\Sel(\EC/K_{\cyc})^\vee\arrow[r]\arrow[d]&0\\
       & \left(H^1(\Q_{\cyc,p},\EC[p^\infty])/\EC(\Q_{\cyc,p})\right)^\vee\otimes \Z_p[G]\arrow[r]&H^1(G_\Sigma(\Q_{\cyc}),\EC[p^\infty])^\vee\otimes \Z_p[G]\arrow[r]&\Sel(\EC/\Q_{\cyc})^\vee\otimes\Z_p[G]\arrow[r]&0
    \end{tikzcd}\]  
It follows from the above discussion that the two vertical arrows are isomorphism.
Thus, we obtain an isomorphism
\[
\Sel(\EC/\Q_{\cyc})^\vee\otimes \Z_p[G]\cong \Sel(\EC/K_{\cyc})^\vee. \qedhere
\]
\end{proof}

In what follows we write $\EC^K$ to denote the quadratic twist of $\EC$; i.e.,  over the quadratic field $K$ the two elliptic curves are isomorphic.
We prove that under Assumption~\ref{ass}, the (dual of) Selmer group of $\EC$ and $\EC^K$ are isomorphic over $\Q_{\cyc}$ as $\Lambda$-modules.

\begin{theorem}
\label{thm:expl-twist}
With notation and assumptions as before, there is an isomorphism
\[
\Sel(\EC/\Q_{\cyc})^\vee\cong \Sel(\EC^K/\Q_{\cyc})^\vee
\]
as $\Lambda$-modules.
In particular, the Iwasawa invariants of $\Sel(\EC/\Q_{\cyc})^\vee$ and $\Sel(\EC^K/\Q_{\cyc})^\vee$ are equal.
Further, assuming that  $\Sha(\EC/\Q_{(n)})[p^\infty]$ and $\Sha(\EC^K/\Q_{(n)})[p^\infty]$ are finite for all $n$, the following isomorphisms hold
\begin{align*}
\Sha(\EC/\Q_{\cyc})[p^\infty]^\vee &\cong\Sha(\EC^K/\Q_{\cyc})[p^\infty]^\vee \\
(\EC(\Q_{\cyc})\otimes \Qp/\Z_p)^\vee&\cong (\EC^K(\Q_{\cyc})\otimes \Qp/\Z_p)^\vee.
\end{align*}
\end{theorem}

\begin{proof}
As before set $G=\Gal(K/\Q) = \langle \iota\rangle$.
Note that 
\[H^1(G_\Sigma(K_{\cyc}),\EC[p^\infty])^{1-\iota}=H^1(G_\Sigma(\Q_{\cyc}),\EC^K[p^\infty]).\]
Under the assumption that $p$ splits in $K$, we have
\[
\left(\prod_{v\in \Sigma_p(K_{\cyc})}\left(H^1(K_{\cyc,v},\EC[p^\infty])/\EC(K_{\cyc,v})\otimes \Qp/\Z_p\right)^\vee\right)^{1-\iota}=\left(H^1(\Q_{\cyc,p},\EC^K[p^\infty])/\EC^K(\Q_{\cyc,p})\otimes \Qp/\Z_p\right)^\vee ,
\]
where (by abuse of notation) $\Q_{\cyc,p}$ is the completion of $\Q_{\cyc}$ at the unique prime above $p$ in $\Q_{\cyc}$.
Invoking the definition of the Selmer group, we observe that
\[
\left(\Sel(\EC/K_{\cyc})^\vee\right)^{1-\iota}=\Sel(\EC^K/\Q_{\cyc})^\vee.
\]
Using Proposition \ref{central-argument} we obtain
\begin{align*}
\Sel(\EC/\Q_{\cyc})^\vee\cong\left(\Sel(\EC/\Q_{\cyc})^\vee\otimes \Z_p[G]\right)^{1-\iota}\cong \left(\Sel(\EC/K_{\cyc})^\vee\right)^{1-\iota}=\Sel(\EC^K/\Q_{\cyc})^\vee.
\end{align*}
This proves the first assertion.

For the rest of the proof we assume that $\Sha(\EC/\Q_{(n)})[p^\infty]$ and $\Sha(\EC^K/\Q_{(n)})[p^\infty]$ are finite for all $n$.
Note that it suffices to prove 
\[
\Sha(\EC/\Q_{\cyc})[p^\infty]^\vee \cong \Sha(\EC^K/\Q_{\cyc})[p^\infty]
^\vee
\]
and the assertion on the Mordell--Weil groups follows immediately.

Let $\mathfrak{G}$ be the operator defined by J.~Lee in \cite{Lee20}.
By Theorem B in \emph{loc. cit.} we obtain
\[
\Sha(\EC/\Q_{\cyc})[p^\infty]^\vee\cong \mathfrak{G}(\Sel(\EC/\Q_{\cyc})^\vee) \cong \mathfrak{G}(\Sel(\EC^K/\Q_{\cyc})^\vee)\cong \Sha(\EC^K/\Q_{\cyc})[p^\infty]^\vee.
\]
We remark that we can apply the aforementioned result since Mazur's Control Theorem holds for $\EC/\Q_{\cyc}$.
\end{proof}

The following corollary will tell us that (under reasonable hypotheses) our assumptions imply that $\EC/\Q$ and the twist $\EC^K/\Q$ have the same Mordell--Weil rank over the base.

\begin{corollary}
\label{cor to main expl-twist}
Suppose that Assumption~\ref{ass} holds.
Further assume that $\Sha(\EC/\Q_{(n)})[p^\infty]$ and $\Sha(\EC^K/\Q_{(n)})[p^\infty]$ are finite for all $n$.
Then $\rk_{\Z}(\EC(\Q_{(n)}))=\rk_{\Z}(\EC^K(\Q_{(n)}))$ for all $n$.
\end{corollary}

\begin{proof}
It is explained in \cite[Remark~2.1.6(2)]{Lee20} that
\[
\EC(\Q_{(n)}) \otimes \Qp/\Zp \longrightarrow (\EC(\Q_{\cyc})\otimes \Qp/\Z_p)^{\Gamma_n}
\]
has finite kernel and cokernel for all $n$.
It is straight forward that
\[
\rk_{\Z}(\EC(\Q_{(n)})) = \rk_{\Zp}((\EC(\Q_{\cyc})\otimes \Qp/\Z_p)^{\vee,\Gamma_n}).
\]
Hence, the same is also true for $\EC^{K}$.
On the other hand, in Theorem~\ref{thm:expl-twist} we proved that 
\[
(\EC(\Q_{\cyc})\otimes \Qp/\Z_p)^\vee \cong (\EC^K(\Q_{\cyc})\otimes \Qp/\Z_p)^\vee.
\]
The result follows.
\end{proof}

\begin{remark}
We observe that this corollary should not be surprising to the reader.
In view of the results in \cite{KR21} (which we have referred to before) for a fixed prime $p\gg 0$, it is expected that `most of the time'
\[
\lambda(\EC/\Q_{\cyc}) = \rk_{\Z} \EC(\Q_{\cyc}) = \rk_{\Z} \EC(\Q)
\]
In other words, there is no rank growth in the cyclotomic $\Zp$-extension `most of the time' when varying over elliptic curves of fixed rank.
{What Theorem~\ref{thm:expl-twist} and Corollary~\ref{cor to main expl-twist} are saying is that under reasonable hypotheses (whose validity will be checked in subsequent sections) $\EC/\Q$ and its twist $\EC^{K}/\Q$ have the same rank.}
Once {we have that $\EC/\Q$ and $\EC^K/\Q$ have the same Mordell--Weil rank} it is natural to expect that their rank will be equal in the $\Zp$-extension.
\end{remark}

\subsection{Discussion on rank 0 elliptic curves over \texorpdfstring{$K$}{}}
\label{sec: rank 0 discussion}
In this paragraph we indicate how the results of this paper are not enough to completely understand the case when $\EC/K$ has Mordell--Weil rank 0.
But our results complement the theorems proven using the Euler characteristic method, thereby completing the story.

Let $\EC$ be a rational elliptic curve that need not satisfy Assumption~\ref{ass}.
If $\EC/K$ is of rank $0$ and $\Sha(\EC/K){[p^\infty]}$ is finite, then $\Sel(\EC/K)$ is finite and both $\EC/\Q$ and $\EC^K/\Q$ have Mordell--Weil rank 0.
The variation of Iwasawa invariants of $\Sel(\EC/K_{\cyc})$ may be studied using the Euler characteristic formula as the tuple $(\EC,p)$ is fixed and $K$ varies via calculations near-identical to \cite[Section~7]{HKR}.
In particular, let $\EC/K$ be an elliptic curve with good ordinary reduction at $p\geq 5$ such that $p$ splits in $K$, $\EC(\Q_p)[p]=\{0\}$, and $p$ does not divide the Tamagawa number of $\EC/\Q$.
Suppose that $K=\Q(\sqrt{\pm d})$ varies such that $\gcd(\Delta_K, pN_{\EC})=1$, where $\Delta_K$ is the fundamental discriminant of $K$.
Then the (algebraic) Iwasawa invariants $\mu(\EC/K_{\cyc})=\lambda(\EC/K_{\cyc})=0$ precisely when $\Sha(\EC/K)[p^\infty]=\{0\}$; see \cite[Conjecture~1.1]{BKLOS21}.
When the conditions are satisfied,
\[
\mu(\EC/\Q_{\cyc})= \mu(\EC^K/\Q_{\cyc}) = 0= \lambda(\EC/\Q_{\cyc}) = \lambda(\EC^K/\Q_{\cyc}).
\]

We now restrict to rational elliptic curves such that $\EC/K$ has rank 0 and Assumption~\ref{ass}(i)-(iii) hold.
Observe that Assumption~\ref{ass}(i) implies $\EC(K)[p] =\{0\}$.
For $p\neq 2$ we now that $\EC(\Q)[p] \oplus \EC^{K}(\Q)[p] \simeq \EC(K)[p]$; thus,
\[
\EC(\Q)[p] = \EC^{K}(\Q)[p] =\{0\}.
\]

\begin{enumerate}[label = (\roman*)]
\item If $\Sha(\EC/\Q)[p^\infty] = \Sha(\EC^K/\Q)[p^\infty] =\{0\}$ then %Assumption~\ref{ass}(v) does not hold but 
via the Euler characteristic argument above
\[
\mu(\EC/\Q_{\cyc})= \mu(\EC^K/\Q_{\cyc}) = 0= \lambda(\EC/\Q_{\cyc}) = \lambda(\EC^K/\Q_{\cyc})
\]
for all sufficiently large $p$ (which are non-anomalous but this is guaranteed by assumptions).
We reiterate that Assumption~\ref{ass}(iv) is not required for this argument.
We also do not require the full force of (i); we only require that $\EC(K_{\fp})[p]=\{0\}$ where $\fp\mid p$ -- let us call this Assumption~\ref{ass}(i').

%\item Now suppose that $\EC/K$ has Mordell--Weil rank 0 and $\EC$ satisfies Assumption~\ref{ass}(i)-(iv).
%By the above argument, both $\EC$ and $\EC^K$ have no $p$-torsion over $\Q$ (and $K$).
%\KM{By Remark~\ref{rank1-case}, Assumption~\ref{ass}(v) is never satisfied and results in this section do not hold as $\EC/\Q$ and $\EC^K/\Q$ have different ranks.}

\item We consider the case where $\EC/K$ has rank 0 and Assumptions~\ref{ass}(i')-(iii) hold.
Suppose that $\Sha(\EC/\Q)[p^\infty]=\{0\}$ and $\Sha(\EC^K/\Q)[p^\infty]\neq \{0\}$.
Our assumptions imply that $\Sha(\EC/K)[p^\infty]\neq \{0\}$.
For $p\gg0$, we can guarantee that 
\[
\mu(\EC/\Q_{\cyc}) = \lambda(\EC/\Q_{\cyc}) =0
\]
but for $\EC^K/\Q$ the Euler characteristic argument yields that \emph{either} $\mu$ or $\lambda$ is positive.
In view of a conjecture of R.~Greenberg we expect that for $p\geq 37$ the $\mu$-invariant is always 0 for all rational elliptic curves so, in fact (heuristically) $\lambda(\EC^K/\Q_{\cyc})>0$.

\item We remind the reader that if $\EC(\Q)[p]=\{0\}$ then it follows from a result of \cite[p.~5051]{Qiu14} that
\[
\# \Sha(\EC/\Q)[p^\infty] \times \# \Sha(\EC^K/\Q)[p^\infty] = \Sha(\EC/K)[p^\infty].
\]
In particular, it can not happen that $\Sel(\EC/\Q) = \Sel(\EC^K/\Q)= \{0\}$ but $\Sha(\EC/K)[p^\infty]\neq \{0\}$.
\end{enumerate}

\subsection*{Overview on the algebraic approach}
Let us summarize what we can achieved using the algebraic approach: under Assumption \ref{ass} we can show that Selmer groups, Shafarevich--Tate groups, and Mordell--Weil groups of $\EC$ and $\EC^K$ are isomorphic along the cyclotomic tower.
The set of primes satisfying conditions (i)-(iii) is a set of positive density and will compute this density in Section \ref{sec: density}.
Let us assume that in addition condition (iv) is satisfied; {this will be studied in more detail in the next section}.
In this case, the fine Shafarevich--Tate group is trivial and by \cite[Theorem~3.4]{wuthrich2007fine}, we conclude that $\Sha(\EC/\Q)[p^\infty]$ is trivial.
Thus, $\Sel(\EC/\Q)$ is non-trivial if and only if $\EC(\Q)$ has rank 1 and we expect that our result holds for half of all twists.% which shows that under assumption (i)-(iv) we have indeed a full analysis of the Shafarevich--Tate groups along the cyclotomic tower.   

\section{Verifying Assumption~\ref{ass}}
\label{sec: verifying}
We provide sufficient conditions on an elliptic curve $\EC/\Q$, a fixed prime $p\geq 5$, and a quadratic field $K$ such that Assumption~\ref{ass} (i)-(iii) are satisfied.
We then check under what additional conditions we can ensure that (iv) is satisfied, as well.
We provide density computations for how many primes $p$ is Assumption~\ref{ass} satisfied, when $K$ and $\EC$ are fixed.

\subsection{Sufficient conditions for \texorpdfstring{Assumption~\ref{ass}}{} to be satisfied}
We first want to verify how often the first three assumptions are satisfied.

\begin{theorem}
\label{thm 4.11}
Let $\EC/\Q$ be an elliptic curve such that 
    \begin{itemize}
        \item $\EC/\Q$ has good ordinary reduction at $p$.
        \item $\EC(\Q_\ell)[p]= \{0\}$ for all $\ell\mid N_{\EC}p$.
    \end{itemize}
    Choose a quadratic field $K$ such that
\begin{itemize}
\item the discriminant of $K$ is only divisible by rational primes $w$ such that $\EC(\Q_w)[p]=\{0\}$.
\item all primes $\ell \mid N_{\EC}p$ are split in $K$.
\end{itemize}
Then Assumption~\ref{ass}\textup{(}i\textup{)}-\textup{(}iii\textup{)} is satisfied.
\end{theorem}

\begin{proof}
Observe that $\ell\mid N_{\EC}p$ is required to be split in $K$ whereas (by definition) the primes $w$ dividing the discriminant of $K$ are ramified in $K$.
Thus, $\gcd(w, N_{\EC})=1$.
Let $w$ be a rational prime dividing the discriminant of $K$ and let $v\mid w$ be the unique prime above it in $K$.
Then $\Q_w(\EC[p])/\Q_w$ is unramified, while $K_v/\Q_w$ is totally ramified.
Thus, $\EC(K_v)[p]=0$.
Therefore condition (i), (ii) and (iii) are satisfied.
\end{proof}

The following remark explains that the conditions imposed on $\EC/\Q$ and $p$ are mild.

\begin{remark}
It is known that the condition $\EC(\Q_\ell)[p]=\{0\}$ for all bad primes $\ell$ is satisfied for all but finitely many $p$.
If $\EC$ is an elliptic curve with (resp. without) complex multiplication then the good ordinary condition is satisfied for 50\% (resp. 100\%) of the primes $p$.
Finally, note that $\EC(\Qp)[p] =\{0\}$ if $p$ is non-anomalous.
\end{remark}

We now want to understand the viability of Assumption~\ref{ass}(iv).
\textit{We assume throughout this discussion that the Shafarevich--Tate group is finite.}

Fix an odd prime $p$ and an elliptic curve $\EC/\Q$ of Mordell--Weil rank at most 1 with no $p$-torsion and such that $\Sha(\EC/\Q)[p^\infty]=\{0\}$.
Recall that we expect 100\% of $\EC$ to have rank at most 1 and by a result of Duke, we know that 100\% of elliptic curves have no torsion.
For a fixed $\EC$, the condition on the Shafarevich--Tate group is expected to hold for all but finitely many $p$.
Now, consider a twist $\EC^K/\Q$ where $K$ is a quadratic field.
If Goldfeld's conjecture is true, then for 100\% of $K$ we know that $\EC^K/\Q$ has rank at most 1.
In other words, for a fixed $\EC$ and 100\% $K$ we can guarantee that $\rk_{\Z}(\EC/K) \leq 2$.
This theoretically makes it possible to achieve that $\Sel^0(\EC/K) =\{0\}$ for a large class of triples $(\EC, K, p)$.

The difficulty is in ensuring that the triviality is indeed achieved.
A priori when $\EC/K$ has Mordell--Weil rank at most 2, then $\Sel^0(\EC/K)$ has $\Zp$-rank either 0 or 1.
But, in our setting we can do better: if $\rk(\EC/\Q)\leq 1$ and $\Sha(\EC/\Q)[p^\infty]=\{0\}$, then we know that $\Sel^0(\EC/\Q)$ is finite.
In fact, under our additional assumption that $\EC(\Q_p)[p]=\{0\}$ it can be show that $\Sel^0(\EC/\Q) =\{0\}$ precisely when the cokernel of $\EC(\Q)\otimes \Zp$ to the $p$-adic completion of $\EC(\Q_p)$ is trivial; see \cite[Theorem~7.1]{wuthrich2007fine}.
It is pertinent to point out that in \cite[Section~10]{Wut05} there are (many) examples where $\Sel^0(\EC/\Q)$ for a rank 1 curve is shown to be trivial but there are also (few) examples where $\Sel^0(\EC/\Q)$ is non-trivial.
Recall from \cite{Qiu14} that $\# \Sha(\EC/K)[p^\infty] = \# \Sha(\EC/\Q)[p^\infty] \times \# \Sha(\EC^K/\Q)[p^\infty]$.
Now for fixed $\EC$ and fixed $p$, there should be many\footnote{On the flip side it is predicted in \cite[Conjecture~1.1]{BKLOS21} that for $K=\Q(\sqrt{d})$ there is a positive proportion of square-free $d$ such that $\Sha(\EC^K/\Q)[p^\infty]$ is non-trivial.} $K$ such that $\Sha(\EC^K/\Q)[p^{\infty}]=\{0\}$, this should ensure that $\Sel^0(\EC^K/\Q)$ is finite.
We have guaranteed that the (fine) Shafarevich--Tate group does not contribute\footnote{We remind the reader that for elliptic curves $\EC/\Q$ of positive rank, the $p$-primary fine Shafarevich--Tate group has the same group structure as $\Sha(\EC/\Q)[p^\infty]$; see \cite[Theorem~3.5]{wuthrich2007fine}.
Even otherwise, non-triviality of $\Sha(\EC/\Q)[p^\infty]$ implies that the $p$-primary fine Shafarevich--Tate group is non-trivial; see \cite[Theorem~3.4]{wuthrich2007fine}} to the fine Selmer group and that $\Sel^0(\EC/K)$ is finite but to guarantee that the fine Mordell--Weil contribution is trivial, requires additional hypotheses.

Set $F=K(\EC[p])$ and assume that $\Gal(F/K)=\GL_2(\mathbb{F}_p)$.
Set $L=K(\zeta_p)$ where $\zeta_p$ denotes a primitive root of unity.
We have a natural inclusion
\[
\Sel^0(\EC[p]/K) \hookrightarrow \Sel_{\Sigma}^0(\EC[p]/L).
\]
We remind the reader that we have shown in Lemma~\ref{isom:different-selmer} that under Assumption~\ref{ass}(i) the fine Selmer group $\Sel^0(\EC[p]/K)$ does not depend on the choice of $\Sigma$.
However, over $L$ the fine Selmer group may still depend on the finite $\Sigma(L)$.
In the remainder of this section (see Proposition~\ref{fine selmer injection}) we will show that
\[
\Sel_{\Sigma}^0(\EC[p]/L) \hookrightarrow \Sel_{\Sigma}^0(\EC[p]/F).
\]
Once this is established, we can provide arithmetic conditions on the number field $F$ to ensure that $\Sel^0(\EC[p]/K)$ is trivial.
Indeed, recall from \cite[Theorem~5.1]{LM16} that
\[
\Sel^0_{\Sigma}(\EC[p]/F) \simeq \Cl_{\Sigma}(F)[p] \oplus \Cl_{\Sigma}(F)[p],
\]
where $\Cl_{\Sigma}(F)$ denotes the $\Sigma$-class group of $F$.
This is a quotient of the class group of $F$ and corresponds to a subfield of the Hilbert class field of $F$ where all primes in $\Sigma$ are completely split.
Therefore, if we require that $p\nmid \#\Cl_{\Sigma}(F)$ then we can ensure that $\Sel^0(\EC[p]/K)$ is trivial.

Define $H:=\Gal(F/L)\simeq\SL_2(\mathbb{F}_p)$.
Consider the following commutative diagram
\[
\begin{tikzcd}
    0\arrow[r]& \Sel^0_{\Sigma}(\EC[p]/L)\arrow[r]\arrow[d,"\alpha"] &H^1(G_\Sigma(L),\EC[p])\arrow[r]\arrow[d,"\beta"]&\displaystyle\prod_{v\in \Sigma(K)}H^1(L_v,\EC[p])\arrow[d,"\gamma"]\\
    0\arrow[r]&\Sel^0_{\Sigma}(\EC[p]/F)\arrow[r]&H^1(G_\Sigma(F),\EC[p])\arrow[r]&\displaystyle\prod_{w\in \Sigma(F)}H^1(F_w,\EC[p])
\end{tikzcd}
\]
Let $Z^1(H,\EC[p])$ be the co-cycles and $B^1(H,\EC[p])$ be the co-boundaries.
\begin{lemma}
    Let $C\subset Z^1(H,\EC[p])$ be the subgroup of co-cycles that map $S=\begin{pmatrix}
        0&-1\\1&0
    \end{pmatrix}$ to zero. Then there is a one-to one correspondence
    \[H^1(H,\EC[p])\leftrightarrow C.\]
\end{lemma}

\begin{proof}
Note that $S^4=\operatorname{Id}_{2\times 2}$ and $H^1(\langle S\rangle ,\EC[p])=0$ since $\gcd(|\langle S \rangle|, |\EC[p]|) =1$.
Let $\Phi\in Z^1(H,\EC[p])$ then there exists an element $\Phi'\in B^1(H,\EC[p])$ such that $\Phi$ and $\Phi'$ coincide on $S$.
Thus, $\Phi-\Phi'$ is the zero map on $\langle S\rangle$.
Therefore, each class in $H^1(H,\EC[p])$ has a representative in $C$. 

Conversely, let $\Phi$ be an element in $C$.
We have to show that $C$ cannot be a coboundary.
Note that 
\[
S-\operatorname{Id}_{2\times 2}=\begin{pmatrix}
        -1&-1\\1&-1
    \end{pmatrix}.
    \]
    This matrix is invertible which implies that there is no point $P\in \EC[p]$ such that $(S-\operatorname{Id}_{2\times 2})P=0$. In particular, $\Phi$ cannot be a coboundary which concludes the proof of the lemma.
\end{proof}

Consider the following two elements
\[
T=\begin{pmatrix}
    1&1\\0&1
\end{pmatrix}\quad T'=\begin{pmatrix}
    1&0\\
    1&1
\end{pmatrix}.
\]
Let $H'=\langle T\rangle$ and $H''=\langle T'\rangle$. Both groups are cyclic of order $p$. 
Furthermore,
\[
TST=\begin{pmatrix}
    1&1\\0&1
\end{pmatrix}\begin{pmatrix} 0&-1\\1&0\end{pmatrix}\begin{pmatrix}
    1&1\\0&1
\end{pmatrix}=\begin{pmatrix}
    1&1\\0&1
\end{pmatrix}\begin{pmatrix}0&-1\\1&1\end{pmatrix}=T'.
\]

\begin{lemma}
\label{cohomology-small groups}
With notation introduced above,
\[
H^1(H',\EC[p])=\frac{\EC[p]}{\mathbb{F}_p\begin{pmatrix}
        1\\0
    \end{pmatrix}}\quad \text{and} \quad  H^1(H'',\EC[p])=\frac{\EC[p]}{\mathbb{F}_p\begin{pmatrix}
        0\\1
    \end{pmatrix}}\]
\end{lemma}

\begin{proof}
Let $P=\begin{pmatrix}
        a\\b
    \end{pmatrix}\in\EC[p]$.
Then 
\[\sum_{k=0}^{p-1}T^k P=\sum_{k=0}^{p-1}\begin{pmatrix}
        a+kb\\b
    \end{pmatrix}=0.
\]
Furthermore since $T-\operatorname{Id}_{2\times 2}=\begin{pmatrix}
        0&1\\0&0
    \end{pmatrix}$, we check that $(T-\operatorname{Id}_{2\times 2})\EC[p]=\mathbb{F}_p\begin{pmatrix}
        1\\0
    \end{pmatrix}$ and the first claim follows. The proof of the second claim is analogously. 
\end{proof}

\begin{lemma}
\label{injectivec}
    There is a natural embedding
    \[H^1(H,\EC[p])\hookrightarrow H^1(H', \EC[p])\times H^1(H'',\EC[p])\]
\end{lemma}

\begin{proof}
Recall the notation from the previous lemma.
Let $\Phi$ be an element in $C$ such that $C\mid _{H'}$ is a coboundary, i.e. $\Phi(T)=\begin{pmatrix}
        a\\0
    \end{pmatrix}$.
    We obtain
    \begin{align*}
        \Phi(T')=\Phi(TST)&=\Phi(T)+T\Phi(ST)=\Phi(T)+T(\Phi(S)+S\Phi(T))=\Phi(T)+TS\Phi(T)\\
        &=\begin{pmatrix}
            a\\0
        \end{pmatrix}+T\begin{pmatrix}
            0\\a
        \end{pmatrix}=\begin{pmatrix}
            a\\0
        \end{pmatrix}+\begin{pmatrix}
            a\\a
        \end{pmatrix}=\begin{pmatrix}
            2a\\a
        \end{pmatrix}.
    \end{align*}
    By Lemma~\ref{cohomology-small groups} this is not a coboundary on $H''$. 
    \end{proof}
\begin{lemma}
\label{inertia}Assume that $\EC/\Q$ has at least one prime $v>2$ of multiplicative reduction with $\vert j(\EC)\vert_v>1$, {where $j(\EC)$ denotes the usual $j$-invariant associated with the elliptic curve}. 
There is a prime in $\Sigma$ such that $T$ is an element in its inertia group. 
\end{lemma}

\begin{proof}
If $\EC/\Q$ has a prime $v$ of multiplicative reduction, there exists a quadratic extension $K'$ such that $\EC/K'$ has split multiplicative reduction.
By \cite[Proposition~5.6.1]{silverman}, the inertia subgroup of $K'_{v'}$ contains an element that acts as $T$ on $\EC[p]$. As $K'/\Q$ is an extension of degree $2$ the same follows for $G_{K_v}$. 
\end{proof}

\begin{proposition}
\label{fine selmer injection}
Assume that $\EC$ has at least one prime $v>2$ of multiplicative reduction $\vert J(\EC)\vert_v>1$.
Then the map
\[
\ker(\beta)\to \ker(\gamma)
\]
is injective.
In particular, 
\[
\Sel_{\Sigma}^0(\EC[p]/L) \hookrightarrow \Sel_{\Sigma}^0(\EC[p]/F).
\]
\end{proposition}

\begin{proof}
By Lemma~\ref{inertia} there exists a prime $v$ such that $T$ lies in the inertia group $I_v$.
As $H'$ is a $p$-Sylow subgroup of $H$, there exists an element $\sigma\in H$ such that $\sigma T\sigma^{-1}=T'$.
Thus, $T'\in I_{\sigma v}$.
If $\sigma\in I_v$, then $I_v$ contains both $H'$ and $H''$.
Lemma~\ref{injectivec} implies that
\[
C \longrightarrow H^1(I_v,\EC[p]) \longrightarrow H^1(H',\EC[p])\times H^1(H'',\EC[p])
\]
is injective.
If $\sigma\notin I_v$ then the same lemma implies
\[
C \longrightarrow H^1(I_v, \EC[p])\times H^1(I_{\sigma v},\EC[p]) \longrightarrow H^1(H',\EC[p])\times H^1(H'',\EC[p])
\]
is injective.
In both cases we obtain that $\ker(\beta)$ injects into $\ker(\gamma)$.
%    \[\ker(\beta)\to \ker(\gamma)\]
%    is injective and 
In particular $\ker(\alpha)=0$.
\end{proof}

As an immediate corollary we obtain:

\begin{corollary}
\label{class number corollary}
If the class number of $F$ is not divisible by $p$, then Assumption~\ref{ass}\textup{(}iv\textup{)} is satisfied.
\end{corollary}

\subsection{Fixing \texorpdfstring{$K$}{} and \texorpdfstring{$\EC$}{}}
\label{sec: fix K and E}
The natural question would be to fix $\EC$ and $p$ and ask for how many imaginary quadratic fields does Assumption~\ref{ass} hold.
In this subsection we ask the reverse question: fixing $K$ and $\EC$, for how many primes $p$ is Assumption \ref{ass} satisfied. Throughout this section we will assume that $\Sha(\EC/\Q)$ is finite for all elliptic curves $\EC$. 

\begin{lemma}
\label{fS is trivial}   
Let $\EC/\Q$ be an elliptic curve and let $K$ be an imaginary quadratic field such that $\EC/\Q$ and $\EC^K/\Q$ both have rank $1$.
Fix a prime $p$ which splits in $K$ and is a prime of good ordinary reduction for $\EC$.
For all but finitely many $p$ such that $\Sha(\EC/\Q)[p^\infty] = \Sha(\EC^K/\Q)[p^\infty] = \{0\}$ and that $\EC(\Q_p)[p]=\{0\}$ we obtain
\[
\Sel^0(\EC/\Q)=\Sel^0(\EC^K/\Q)=\{0\}.
\]
In particular, $\Sel^0(\EC/K) = \{0\}$.
\end{lemma}

%\DK{\sf \textbf{[In this lemma we do not specify CM vs non-CM. If $\EC$ is a prime of good ordinary reduction for $\EC$, is it necessary that it is a prime of good ordinary reduction for $\EC^K$? I think we need good ordinary reduction for both elliptic curves because we are repeating the argument for $\EC^K$]}} \KM{True, but if $p$ splits in $K$ the curves $\EC$ and $\EC^K$ are isomorphic over $\Q_p$. }

\begin{proof}
We know from \cite[Corollary~7.3]{wuthrich2007fine} that $\Sel^0(\EC/\Q)$ and $\Sel^0(\EC^K/\Q)$ are both finite.
Moreover, the fine Shafarevich--Tate group is a subgroup of the (usual) Shafarevich--Tate group; thus, if the fine Selmer group is non-trivial it must have a contribution from the fine Mordell--Weil group. 

We provide an argument for $\EC/\Q$, the argument can then be repeated for $\EC^K/\Q$.
{Recall that since $p$ splits in $K$, the curves $\EC$ and $\EC^K$ are isomorphic over $\Q_p$.
It follows that if $p$ is a prime of good ordinary reduction for $\EC/\Q$, the same is true for $\EC^K/\Q$.}

We have the following exact sequence
\[
0\longrightarrow \Sel^0(\EC/\Q)\longrightarrow \EC(\Q)\otimes \Q_p/\Z_p\longrightarrow \EC(\Q_p)\otimes \Q_p/\Z_p\longrightarrow 0,
\]
where the right zero comes from the fact that $\EC(\Q)$ and $\EC(\Q_p)$ are both divisible and of corank $1$.
In order to show that $\Sel^0(\EC/\Q)=\{0\}$ ({as discussed previously}) we need to understand the cokernel of
\[
\EC(\Q) \otimes \Z_p \longrightarrow \EC(\Q_p) \otimes \Z_p.
\]
Let $Q\in\EC(\Q)$ be a point generating $\EC(\Q)/\EC(\Q)_{\textup{tor}}$.
Let $m$ be an integer coprime to $p$ such that $m Q\in \EC_1(\Q_p)$, where $\EC_1(\Q_p)$ is the kernel of the reduction of $\EC(\Q_p)$ modulo $p$.
{Recall that this is the same as the formal group $\widehat{\EC}(\mathbb{Q}_p)$.}
%\DK{\sf \textbf{[Is this different from $\widehat{\EC}(\mathbb{Q}_p)$?]}}\KM{no, they are the same}
Let $(x,y)=Q$ and $m Q=(x',y')$ be the coordinates in $\Q$.

\bigskip

\noindent 
\textit{Claim:}
If $p\nmid v_p(x)$ then the image of $Q$ in $\widehat{\EC}(\Q_p)$ is not a $p$-th power.

\bigskip

\noindent
\textit{Justification:}
Note that $v_p(x)=v_p(x')$ and $v_p(y)=v_p(y')$.
If $mQ$ is a $p$-th power in $\widehat{\EC}(\Q_p)$, then $v_p(x')$ and $v_p(y')$ are divisible by $p$.

As $Q$ is a rational point this implies that $\EC(\Q)\to \EC(\Q_p)$ is surjective.
\end{proof}

\begin{remark}
The converse of Lemma~\ref{fS is trivial} is also true.
Fix a prime $p$ of good ordinary reduction which splits in $K$.
If $\EC/\Q$ is {an elliptic curve} of {Mordell--Weil} rank $1$, $\EC(\Q_p)[p]=0$, and $\Sha(\EC/\Q)[p^\infty]\neq 0$, then $\Sel^0(\EC/\Q)$ is non trivial. The proof of Lemma~\ref{fS is trivial} shows that the fine Mordell--Weil group is trivial and that 
\[
\EC(\Q) {\otimes \Z_p} \longrightarrow \EC(\Q_p) {\otimes \Z_p}
\]
is surjective.
%\KM{Not really: I take the same exact sequence asin the above proof but with the fine Mordell-Weyl group instead of the fine Selmer. Then our arguments with the valuations show that the fine Mordell-Weyl group has to vanish. }
Let now $c$ be a non-trivial element in $H^1(G_\Sigma(\Q),\EC[p^\infty])$ that has a non trivial image in $\Sha(\EC/\Q)[p^\infty]$.
The restriction of $c$ to $G_{\Q_p}$ lies in $\EC(\Q_p)\otimes \Q_p/\Z_p$.
Then there exists a point $P\in \EC(\Q)$ and a non-trivial integer $n$ such that $P\otimes p^{-n}$ has the same restriction as $c$.
Thus, $P\otimes p^{-n}-c\in \Sel^0(\EC/\Q)$ is a non-trivial element. 
\end{remark}

\begin{theorem}
\label{theorther way round}
Let $\EC/\Q$ be an elliptic curve and let $K$ be an imaginary quadratic field such that $\EC/\Q$ and $\EC^K/\Q$ both have Mordell--Weil rank $1$.
Then Assumption~\ref{ass} holds for {density one} primes of {good} ordinary reduction that split in $K$.
\end{theorem}

\begin{proof}
By our hypothesis on the rank of the elliptic curves $\EC/\Q$ and $\EC^K/\Q$, it is clear that Assumption~\ref{ass}(v) holds.
Further, we only vary over those primes $p$ which split in $K$ and such that $\EC$ (and {hence} $\EC^K$) has good ordinary reduction at $p$.
This forces Assumption~\ref{ass}(ii)--(iii).
Let us denote this set of primes by $\mathcal{S}$.
Recall that by Chebotarev Density Theorem, half of all the rational primes split in $K$.
Also, if $\EC/\Q$ is a non-CM (resp. CM) elliptic curve then it has good ordinary reduction at density one (resp. half) rational primes.

Let $q$ be a prime of bad reduction for either $\EC$ or $\EC^K$.
Let $v\mid q$ be a prime in $K$.
Then $\EC(K_v)_{\operatorname{tor}}$ is a finitely generated abelian group. 
In particular, $\EC(K_v)[p]=0$ for all but finitely many $p$.

We first consider the case that $\EC(K)$ has a rational torsion point, say $P$ of order $r$.
By \cite[Lemma~A.5]{MR05} we know that for every rational prime $p> 5$ , $p \nmid r$, and every prime $v\mid p$ of $K$ of degree one where $\EC$ has good reduction, $\EC(K_v)[p] = \EC(\Q_p)[p]$ is trivial.
We have shown that when $\EC(K)$ has a non-trivial torsion point then Assumption~\ref{ass}(i) holds for all but finitely many primes $p\in \mathcal{S}$.

Now suppose that $\EC(K)$ has no rational torsion point.
Consider the short exact sequence
\[
0 \longrightarrow \widehat{\EC}(p\Z_p) \longrightarrow \EC(\Q_p) \longrightarrow \widetilde{\EC}(\mathbb{F}_p) \longrightarrow 0.
\]
Since $\widehat{\EC}(p\Z_p)$ is isomorphic to $\Z_p$ as additive groups, it follows that $\EC(\Q_p)[p]$ is non-trivial if and only if the above exact sequence splits \emph{and} $p$ is a prime of anomalous good reduction; see for example \cite[Section~3.1.1]{kim2026structure}.
Recall that a prime $p$ of good ordinary reduction is called anomalous if $p\mid \#\widetilde{\EC}(\mathbb{F}_p)$ (or equivalently $a_p(\EC) \equiv 1\pmod{p}$ where $a_p(\EC)$ is the trace of the Frobenius).
Thus, if $\EC$ has non-anomalous good (ordinary) reduction at $p$, then $\EC(\Q_p)[p]$ is trivial.
We know that density zero primes of good ordinary reduction are anomalous; see \cite[Proposition~5.1]{greenberg-elliptic}.
We have shown that when $\EC(K)_{\operatorname{tors}}$ is trivial then Assumption~\ref{ass}(i) holds for all $p\in \mathcal{S}$ outside a set of density zero.
In view of \cite[Conjecture~1.1 with $d=1$]{DW08}, we expect that $\EC(\Q_p)[p]$ is in fact non-trivial for only finitely many primes $p$.

Since we are assuming finiteness of the Shafarevich--Tate group, we expect that for all but finitely many primes $\Sha(\EC/\Q)[p^\infty] = \Sha(\EC^K/\Q)[p^\infty] = \{0\}$.
Thus, in view of Lemma~\ref{fS is trivial} we observe that Assumption~\ref{ass}(iv) holds for all but finitely many primes $p\in \mathcal{S}$.
\end{proof}

\section{Comparing Densities }
\label{sec: density}
We have seen in Section~\ref{sec: fix K and E} that fixing $\EC$ and $K$ basically allows us to take $100\%$ of primes $p$ of good ordinary reduction that split in $K$.
In the present section we try to determine which proportion of imaginary quadratic fields $K$ can be chosen if $\EC$ and $p$ are fixed.

\subsection{Density computations for results arising from \texorpdfstring{$p$}{}-adic \texorpdfstring{$L$}{}-function approach}
\label{sec:den-l-function}
In this section we return to an elliptic curve $\EC/\Q$ that satisfy the assumptions of Theorem~\ref{central-thm-analytic}. 
Given a possible twist $-n_1$ such that Iwasawa invariants of $\EC^{(-n_1)}$ vanish we want to find $n_2$ such that 
\begin{itemize}
    \item $n_2$ is square-free
    \item $n_2$ is coprime to $4\ell p$
    \item $n_1/n_2\in \Q_q^2$ for all $q\mid 4\ell p$.
    \item $h(-n_2)$ is not divisible by $p$. 
\end{itemize}
The third condition is satisfied by $1/8$ of all $n_2$.
The Wiener--Ikehara theorem gives that for the number of square-free positive integers co-prime to $4\ell p$ is asymptotic to
\[
\frac{6}{\pi^2}\prod_{p\mid 4\ell p} \frac{p}{p+1}.
\]
%https://math.stackexchange.com/questions/1343148/counting-square-free-numbers-co-prime-to-m

The Cohen--Lenstra heuristic predicts that the last condition is satisfied by a proportion of 
\[
\prod_{j\geq 1}(1-p^{-j}).
\]
This conjecture is still open.
For $p>3$ we have the following lower bound \cite{kohen-ono}
\[
\vert\{-X<D<0, p\nmid h(-d)\}\vert \ge \left(\frac{2(p-2)}{\sqrt{3}(p-1)}-\varepsilon\right)\frac{\sqrt{X}}{\log(X)}.
\]
For $p=3$, it is know by the work of H.~Davenport and H.~Heilbronn (see \cite{davenport-heilbronn}) that a positive proportion of imaginary quadratic fields satisfy the condition that $3$ does not divide $h(-d)$.

\subsection{Density computations for results arising from the explicit approach}
Let $\EC/\Q$ be a non-CM elliptic curve with good ordinary reduction at a fixed prime $p$.
Suppose that the mod-$p$ representation of $\EC$ is surjective; this is always guaranteed if $p$ is large enough.
Recall that the maximal abelian extension of $\Q$ contained in $\Q(\EC[p])$ is $\Q(\zeta_p)$.
Moreover, the quadratic extension contained in $\Q(\zeta_p)$ is
\[
\begin{cases}
    \Q(\sqrt{p}) & \text{ if } p \equiv 1\pmod{4}\\
    \Q(\sqrt{-p}) & \text{ if } p \equiv 3\pmod{4}.
\end{cases}
\]

As is usual, write $\pi(x)$ denote the prime counting function up to $x$.
Let $\varepsilon \in \{\pm\}$.
The purpose of this section is to compute the following density
\[
\alpha_{\varepsilon} = \lim_{x\rightarrow\infty} \frac{\# \{\ell < x \text{ a prime} \colon \EC(\Q_{\ell})[p] = \{0\} \text{ and } q \text{ splits in } \Q(\sqrt{\varepsilon \ell})\text{ for all }q\mid pN_{\EC} \}}{\pi(x)}.
\]
For every $q\mid p\cdot N_{\EC}$ we define the quadratic field
\[
\mathcal{F}_q=\begin{cases}
    \Q(\sqrt{-q})\quad &q\equiv 3\pmod{4}\\
    \Q(\sqrt{q})\quad &q\equiv 1\pmod{4}.
\end{cases}
\]
Set $\mathcal{F}$ to denote the compositum of all quadratic extensions $\mathcal{F}_q$ where $q\mid N_{\EC}p$.
The condition that $q$ splits in $\Q(\sqrt{\varepsilon \ell})$ is equivalent to $\left(\frac{\varepsilon \ell}{q}\right)=1$. 
When $q\equiv 1\pmod 4$, this condition can be phrased as
\[
\left(\frac{\varepsilon}{q}\right)\left(\frac{\ell}{q}\right) = \left(\frac{\varepsilon}{q}\right)\left(\frac{q}{\ell}\right)=1.
\]
On the other hand, when $q\equiv 3\pmod 4$ we can reformulate the condition as
\[
\left(\frac{\varepsilon}{q}\right)\left(\frac{\ell}{q}\right)  = \left(\frac{\varepsilon}{q}\right)\left(\frac{-q}{\ell}\right)=1.
\]

Suppose that $\varepsilon=-1$.
If we are in the case that $q\equiv 1 \pmod{4}$ then we require
\[
\left(\frac{\varepsilon}{q}\right)\left(\frac{q}{\ell}\right) = 1 \cdot \left(\frac{q}{\ell}\right) =1.
\]
Equivalently, we are looking for primes $\ell$ that are split in $\Q(\sqrt{q})$ for all $q\mid N_{\EC}p$. 
On the other hand, if we are in the case that $q\equiv 3 \pmod{4}$ then we require
\[
\left(\frac{\varepsilon}{q}\right)\left(\frac{-q}{\ell}\right) = -1 \cdot \left(\frac{-q}{\ell}\right) =1.
\]
Equivalently, we are looking for primes $\ell$ that are inert in $\Q(\sqrt{-q})$ for all $q\mid N_{\EC}p$.
These two conditions define $\textup{Frob}_{\ell}\in \Gal(\mathcal{F}/\Q)$.

Now, suppose that $\varepsilon=1$.
An identical calculation to the one above shows that we are looking for elements that are totally split in $\mathcal{F}$.

\subsubsection{}
Consider the case that $\mathcal{F}\cap \Q(\EC[p]) = \Q$.
In this case, the two events are independent.
Let $e=0$ if $N_{\EC}$ is odd and $e=1$ if $N_{\EC}$ is even.
If $N_{\EC}$ has $s$ many (distinct) prime factors then by using a Chebotarev density argument we know that
\[
\lim_{x\rightarrow\infty} \frac{\# \{\ell < x \text{ a prime} \colon q \text{ splits in } \Q(\sqrt{\varepsilon\ell})\text{ for all }q\mid pN_{\EC} \}}{\pi(x)} = \frac{1}{2^{s+1+e}}.
\]
Indeed, if $N_{\EC}$ is odd, this is just the condition to be split in $\mathcal{F}$ which is an extension of degree $2^{s+1}$ over $\Q$.
If $N_{\EC}$ is even, then $\mathcal{F}$ is an extension of degree $2^{s-1}$ over $\Q$.
For $\varepsilon\in \{+,-\}$, to ensure that $2$ splits in $\Q(\sqrt{\varepsilon\ell})$ we need the additional assumption that $\ell\equiv \varepsilon \pmod{8}$, which holds for $1/4$ of all primes.

On the other hand,
\begin{align*}
\lim_{x\rightarrow\infty} \frac{\# \{\ell < x \text{ a prime} \colon \EC(\Q_{\ell})[p] = \{0\}\}}{\pi(x)}
&= \lim_{x\rightarrow\infty} \frac{\# \{\ell < x \text{ a prime} \colon a_{\ell}(\EC) \not\equiv \ell+1\pmod{p}\}}{\pi(x)} \\
& = 1 - \frac{p^2 -2}{(p-1)^2(p+1)}.
\end{align*}
For the first equality we are using the fact that $\EC(\Q_{\ell})[p] = \{0\}$ is equivalent to 1 not being an eigenvalue of $\Frob_{\ell}$.
For the final equality we use the results proven in \cite[Section~3.2]{HK23}.

The above discussion can be summarized as follows.

\begin{proposition}
\label{prop: twist density when p is 1 mod 4}
Let $\EC/\Q$ be a non-CM elliptic curve with good ordinary reduction at a fixed prime $p$ such that the mod-$p$ Galois representation is surjective.
Assume that $\mathcal{F}\cap \Q(\EC[p])=\Q$.
Suppose that the number of distinct primes dividing $N_{\EC}$ is $s$.
Then
\[
\lim_{x\rightarrow\infty} \frac{\# \{\ell < x \text{ a prime} \colon \EC(\Q_{\ell})[p] = \{0\} \text{ and } q \text{ splits in } \Q(\sqrt{\varepsilon\ell})\text{ for all }q\mid pN_{\EC} \}}{\pi(x)} = \frac{1}{2^{s+1+e}}\left( 1 - \frac{p^2 -2}{(p-1)^2(p+1)}\right).
\]
\end{proposition}

\subsubsection{}

Now we consider the case when $\mathcal{F}\cap \Q(\EC[p]) \neq \Q$.
In this case the two conditions of triviality of $p$-torsion of $\EC/\Q_{\ell}$ and the splitting of $p$ in $\Q(\sqrt{\varepsilon\ell})$ are not independent of each other.
However, splitting of $q\mid N_{\EC}$ in $\Q(\sqrt{\varepsilon\ell})$ is (still) independent of splitting of $p$ in $\Q(\sqrt{\varepsilon\ell})$ and $\EC(\Q_{\ell})[p]=\{0\}$.

The main result is parallel to Proposition~\ref{prop: twist density when p is 1 mod 4}.
More precisely, we prove the following.

\begin{proposition}
\label{prop: twist density when p is 1 mod 4 real}
Let $\EC/\Q$ be a non-CM elliptic curve with good ordinary reduction at a fixed prime $p$ such that the mod-$p$ Galois representation is surjective.
Assume that $\mathcal{F}\cap \Q(\EC[p]) \neq  \Q$.
Suppose that the number of distinct primes dividing $N_{\EC}$ is $s$.
Then
\[
\lim_{x\rightarrow\infty} \frac{\# \{\ell < x \text{ a prime} \colon \EC(\Q_{\ell})[p] = \{0\} \text{ and } q \text{ splits in } \Q(\sqrt{\varepsilon\ell})\text{ for all }q\mid pN_{\EC} \}}{\pi(x)} = \frac{1}{2^{s+1+e}}\left( 1 - \frac{p^2 -3}{(p-1)^2(p+1)}\right).
\]
\end{proposition}

The idea of the proof is similar to the previous one.
We sketch the details below.
In view of the discussion on the independence of events, we can reduce the problem to computing
\begin{align*}
\label{what we want to count in p 3 mod 4 case}
&\lim_{x\rightarrow\infty} \frac{\# \{\ell < x \text{ a prime} \colon \EC(\Q_{\ell} )[p] = \{0\} \text{ and }p \text{ splits in } \Q(\sqrt{\varepsilon\ell})\}}{\pi(x)} \\
&= \lim_{x\rightarrow\infty} \frac{\# \{\ell < x \text{ a prime} \colon \  a_{\ell}(\EC) \not\equiv \ell+1 \ (\text{mod }{p}) \text{ and } \ell \text{ has order dividing $\frac{p-1}{2}$ in }\mathbb{F}_p^\times\}}{\pi(x)} \\
& = 1 - \lim_{x\rightarrow\infty} \frac{\# \{\ell < x \text{ a prime} \colon  a_{\ell}(\EC) \equiv \ell+1 \ (\text{mod }{p}) \text{ or } \ell \text{ has order divisible by $2^{v_2(p-1)}$ in }\mathbb{F}_p^\times\}}{\pi(x)} \\
& = 1 - \left( \frac{p^2 -2}{(p-1)^2(p+1)} + \frac{1}{2} - \frac{1}{2(p-1)}\right)\\
& = \frac{1}{2} - \frac{2p^2 - 4 - (p^2-1)}{2(p-1)^2(p+1)} = \frac{1}{2}\left( 1 - \frac{p^2 - 3}{(p-1)^2(p+1)}\right).
\end{align*}

The above calculation follows from the following observations; see \cite[Section~3.2]{HK23} for details:
\begin{align*}
\lim_{x\rightarrow\infty} \frac{\# \{\ell < x \text{ a prime} \colon a_{\ell}(\EC) \equiv \ell+1 \ (\text{mod }{p})\}}{\pi(x)} & = \frac{p^2 -2}{(p-1)^2(p+1)} \\
\lim_{x\rightarrow\infty} \frac{\# \{\ell < x \text{ a prime} \colon \ell \text{ has order divisible by $2^{v_2(p-1)}$ in $\mathbb{F}_p^\times$}\}}{\pi(x)} & = \frac{1}{2} \\
\lim_{x\rightarrow\infty} \frac{\# \{\ell < x \text{ a prime} \colon  a_{\ell}(\EC) \equiv \ell+1 \ (\text{mod }{p}) \text{ and } \ell\text{ has order divisible by $2^{v_2(p-1)}$ in $\mathbb{F}_p^\times$} \}}{\pi(x)} &= \frac{p(p+1)(p-1)}{2p(p-1)^2(p+1)}\\
& = \frac{1}{2(p-1)}.
\end{align*}

Note that a fixed (odd) prime $q\neq \ell$ splits in either $\Q(\sqrt{\ell})$ or $\Q(\sqrt{-\ell})$ but not both.
We now record our main density result which follows immediately from the above propositions.

\begin{theorem}
\label{density result}
Let $\EC/\Q$ be a non-CM elliptic curve with good ordinary reduction at a fixed prime $p$ such that the mod-$p$ Galois representation is surjective.
Suppose that the number of distinct primes dividing $N_{\EC}$ is $s$ and set $e \equiv N_{\EC} \pmod{2}$.
Then
\[
\lim_{x \to \infty}\frac{ \# \{ \ \abs{\ell} < x \colon K = \Q(\sqrt{\pm \ell}) \text{ satisfying Assumption}~\ref{ass}(i)-(iii)\}}{\pi(x)} = \frac{1}{2^{s+1+e}}\left( 2 - \frac{2p^2 -5}{(p-1)^2(p+1)}\right).
\]
In particular, as $p\to \infty$,
\[
\lim_{x \to \infty}\frac{ \# \{ \ \abs{\ell} < x \colon K = \Q(\sqrt{\pm \ell}) \text{ satisfying Assumption}~\ref{ass}(i)-(iii)\}}{\pi(x)} \to \frac{1}{2^{s+e}} > 0.
\]
\end{theorem}

\begin{remark}
A standard argument using the Tauberian theorem (see \cite[Theorem~7.1]{Wood} or \cite{Narkiewicz}) implies that
\[
\lim_{x \to \infty}\frac{ \# \{ \ \abs{d} < x \colon d \text{ is square-free and } K = \Q(\sqrt{\pm d}) \text{ satisfies Assumption}~\ref{ass}(i)-(iii)\}}{\#\{ \ \abs{d} < x : d \text{ is square-free} \}} = 0.
\]
\end{remark}

\bibliographystyle{amsalpha}
\bibliography{references}

\end{document}